\documentclass[a4paper, 11pt]{article}

\usepackage{amsmath,amssymb,amsthm}
\usepackage{hyperref}
\usepackage{cases}
\usepackage{mathrsfs}
\usepackage{amsfonts}
\allowdisplaybreaks[4]
\usepackage{amscd}
\usepackage[all]{xy}
\usepackage{graphicx}
\usepackage[top=30truemm,bottom=30truemm,left=25truemm,right=25truemm]{geometry}
\usepackage{xcolor}

\usepackage{url}

 \makeatletter
    
    \@addtoreset{equation}{section}
  \makeatother

\DeclareMathOperator{\vol}{vol}

\DeclareMathOperator{\Ker}{Ker}
\DeclareMathOperator{\Hom}{Hom}

\DeclareMathOperator{\Homo}{H}

\DeclareMathOperator{\Res}{Res}
\DeclareMathOperator{\PV}{PV}
\DeclareMathOperator{\rest}{rest}
\DeclareMathOperator{\reg}{reg}
\DeclareMathOperator{\Cone}{Cone}
\DeclareMathOperator{\Exp}{Exp}
\DeclareMathOperator{\dvol}{dvol}

\newcommand{\barsigma}{\overline{\sigma}}
\newcommand{\PP}{\mathbb{P}}
\newcommand{\R}{\mathbb{R}}

\newcommand{\C}{\mathbb{C}}
\newcommand{\DD}{\mathcal{D}}

\newcommand{\Q}{\mathbb{Q}}
\newcommand{\Z}{\mathbb{Z}}
\newcommand{\T}{\mathbb{T}}
\newcommand{\EE}{\mathcal{E}}

\newcommand{\ii}{\sqrt{-1}}
\newcommand{\s}{\sigma}

\newcommand{\bs}{\barsigma}

\newcommand{\ve}{\varepsilon}
\newcommand{\dbar}{\bar{\partial}}

\def\mychange#1{{#1}}

\theoremstyle{plain}
\newtheorem{thm}{Theorem}[section]
\newtheorem{lem}[thm]{Lemma}
\newtheorem{prop}[thm]{Proposition}
\newtheorem{cor}[thm]{Corollary}
\newtheorem{example}[thm]{Example}

\newtheorem{rem}[thm]{Remark}

\def\change#1{{ #1}}

\title{Localization formulas of cohomology intersection numbers}
\author{Saiei-Jaeyeong Matsubara-Heo}

\begin{document}
\date{}
\maketitle

\begin{abstract}
We revisit the localization formulas of cohomology intersection numbers associated to a logarithmic connection. The main contribution of this paper is threefold: we prove the localization formula of the cohomology intersection number of logarithmic forms in terms of residue of a connection; we prove that the leading term of the Laurent expansion of the cohomology intersection number is Grothendieck residue when the connection is hypergeometric; and we prove that the leading term of stringy integral discussed by Arkani-Hamed, He and Lam is nothing but the self-cohomology intersection number of the canonical form.
\end{abstract}

\section{Introduction}
Hypergeometric integrals appear broadly in applied sciences in disguise: it appears as a normalizing constant of a discrete or continuous probability distribution in statistics (\cite{GKTT}, \cite{HGM}); it plays an important role in quantum field theory as a dimensionally regularized Feynman integral (\cite{IBP},\cite{MM}).
Numerical evaluation of hypergeometric integral is a fundamental problem in these areas.
Recently, it comes to light that exact evaluation of {\it cohomology intersection number} can contribute to such a problem (\cite{GKTT},\cite{MM}).
The key property of the cohomology intersection number is that one can compute it only from local information although the definition is a priori global.
In this paper, we revisit two types of localization formulas of cohomology intersection number: {\it stationary phase formula} (\cite{M},\cite{MP}) and {\it residue formula} (\cite{CM},\cite{Matsu}).
As a byproduct, we discuss {\it stringy integral} (\cite{AH}).
Namely, we prove that the leading term of stringy integral is nothing but the self-cohomology intersection number of the {\it canonical form} (\cite{AHP}).

A (regular holonomic) hypergeometric integral takes the following form:
\begin{equation}\label{eqn:1.1}
I_{\omega}(z;\alpha):=\int_\Gamma f_1^{\alpha_1}\cdots f_m^{\alpha_m}\omega.
\end{equation}
Here, $f_1,\dots,f_m$ are nowhere vanishing regular functions on a smooth complex affine variety $U=U_z$ smoothly fibered over a base space of $z$-variables, $\alpha_i$ is a complex parameter, $\omega$ is a differential form on $U$ of degree $n=\dim_{\C} U$ and $\Gamma$ is a suitable integration cycle.
It is natural to regard $\omega$ as a representative of a cohomology class of an algebraic de Rham cohomology group $\Homo^n(\alpha):=\mathbb{H}^n(U;(\Omega_{U}^\bullet,\nabla))$ with $\nabla:=d_x+\sum_{i=1}^m\alpha_id_x\log f_i\wedge$.
Deducing a Pfaffian system that the hypergeometric integral (\ref{eqn:1.1}) is subject to is equivalent to writing down the connection matrix of the Gau\ss-Manin connection with respect to a frame $\{ [\omega_i]\}_i$ of $\Homo^n(\alpha)$.
It is a fast derivation of a Pfaffian system that a numerical evaluation of hypergeometric integral is based on.
Let us set $I(z;\alpha)={}^t(I_{\omega_1}(z;\alpha),\dots)$ and let $d_zI=AI$ be the associated Pfaffian system.
The {\it holonomic gradient method} computes $I(z;\alpha)$ at a large point $z_{large}$ from a value at a small point $z_{small}$ employing a suitable discritization of $d_zI=AI$(\cite{HGM}).
A difference analogue of this method is called {\it difference holonomic gradient method} in \cite{GKTT}.
Namely, we consider an integral difference operator $\s:\alpha\mapsto \alpha+{\bf k}$ and a difference version of Pfaffian system $\s I=BI$.
Then the value $I(z;\alpha+N{\bf k})$ for a large integer $N$ can be obtained from $I(z;\alpha)$ by an iterative multiplication by $B$.
Although the action of Gau\ss-Manin derivative or that of the difference operator $\s$ on the cohomology group $\Homo^n(\alpha)$ is easy to describe, the computation of the connection matrices $A,B$ requires some efforts.

A useful tool of deriving connection matrices is the cohomology intersection form.
Under a suitable assumption, we have a canonical isomorphism $\mathbb{H}_c^*(U^{an};(\Omega^\bullet_{U^{an}},\nabla^{an}))\tilde{\rightarrow}\mathbb{H}^*(U^{an};(\Omega^\bullet_{U^{an}},\nabla^{an}))$.
Combining this canonical isomorphism with Poincar\'e duality and the comparison isomorphism of Deligne-Grothendieck (\cite[Corollaire 6.3]{Del}), we have a perfect bilinear pairing $\langle\bullet,\bullet\rangle_{ch};\Homo^n(-\alpha)\otimes_\C\Homo^n(\alpha)\rightarrow\C$ called the cohomology intersection form.
Note that any element $\omega_\pm\in\Homo^0(U,\Omega_U^n)$ naturally defines a cohomology class $[\omega_\pm]\in\Homo^n(\pm\alpha)$.
Its value $\langle[\omega_-],[\omega_+]\rangle_{ch}$ for some $[\omega_\pm]\in\Homo^n(\pm\alpha)$ is called the cohomology intersection number.
For any linear endomorphism $T$ of $\Homo^n(\alpha)$, its representation matrix with respect to a given basis can be calculated in terms of $\langle\bullet,\bullet\rangle_{ch}$.
In this paper, we revisit two types of formulas of the cohomology intersection number $\langle[\omega_-],[\omega_+]\rangle_{ch}$ in terms of the representatives $\omega_\pm$: {stationary phase formula} (\cite{M}, \cite{MP}) and {residue formula} (\cite{CM},\cite{Matsu}).
Although these formulas are widely used in practice, we aim to provide precise statements of them as well as their proofs.

Let us first recall the stationary phase formula.
To formulate it, we need to introduce a large parameter $\tau$ and consider a scaling of the cohomology group $\Homo^n(\tau\alpha)$.
An important assumption is that the critical points of the phase function (or {\it log-likelihood function}) $F_\alpha=\sum_{i=1}^m\alpha_i\log f_i$ are all isolated and $F_\alpha$ is smooth at infinity.
This is a weak form of Varchenko's conjecture (\cite{Var}) and we will provide a discussion when this is valid in \S\ref{sec:2.2}.
At this stage, we remark that the assumption is true when $U$ is a smooth closed subvariety of an algebraic torus, $f_i$ are coordinate projections restricted to $U$ and $\alpha$ is generic (\cite{Huh}).
Under this assumption, the algebraic de Rham complex is related to the Koszul complex $(\Omega_U^\bullet,dF\wedge)$ associated to $dF$ at the large scale limit $|\tau|\rightarrow\infty$.
As the global cohomology of the Koszul complex has a natural duality pairing in terms of Grothendieck residue, it is natural to expect that it is the leading term of the cohomology intersection number.
The stationary phase formula takes the following form of which the precise formulation is provided in \S\ref{sec:2.2}.

\begin{thm}[Theorem \ref{thm:2}, stationary phase formula]\label{thm:Intro1}
Assume the conditions $(*)$ and ${\rm (generic)}$ of \S\ref{sec:2.2}. For any $\omega_\pm\in\Homo^0(U;\Omega^n_U)$, one has a Laurent expansion
\begin{equation}\label{eqn:AEIntro}
\frac{\langle[\omega_-],[\omega_+]\rangle_{ch}}{(2\pi\ii)^n}=\tau^{-n}\sum_{k\geq 0}K^{(k)}(\omega_-,\omega_+)\tau^{-k}\ \ (|\tau|\rightarrow+\infty)
\end{equation}
of which the leading term $K^{(0)}$ is the Grothendieck residue.
\end{thm}


\begin{cor}[Corollary \ref{cor:2.7}]\label{cor:Intro1}
Assume the conditions $(*)$ and {\rm (generic)} in \S\ref{sec:2.2}.
Then, any basis of $\Homo^n\left(U;(\Omega^\bullet_U,dF\wedge)\right)$ gives rise to that of the algebraic de Rham cohomology group $\Homo^n(\tau\alpha)$  if $\tau$ is chosen generically. 
\end{cor}

\noindent
Theorem \ref{thm:Intro1} is a variant of a well-known result in singularity theory (\cite[\S4, Th\'eor\`eme]{Pham}, \cite[\S4]{KSaito}).
Indeed, the proof employed in this paper is close in spirit to that of \cite[\S4, Th\'eor\`eme]{Pham}.
A subtlety in the statement of Theorem \ref{thm:Intro1} is that the cohomology intersection number in (\ref{eqn:AEIntro}) is taken in various cohomology groups $\Homo^n(\tau\alpha)$.
The conditions $(*)$ and {\rm (generic)} enable us to define a family version of the cohomology intersection number which will be discussed in detail in \S\ref{sec:2.2} and \S\ref{sec:3}.
Corollary \ref{cor:Intro1} indicates an algorithm of computing a basis of the algebraic de Rham cohomology group $\Homo^n(\alpha)$ in terms of commutative algebra, whereas usual algorithm relies on Weyl algebra (cf. \cite{HNT}).
We remark that the assumptions $(*)$ and $({\rm generic})$ are made so that the phase function $F_\alpha$ has only isolated zeros and is smooth at infinity.
These conditions are described using an extrinsic data, i.e., it is formulated with the aid of a compactification $X$ of $U$ (see \S\ref{sec:2.2}).
Recently, a closely related notion of {\it critical slopes} is introduced and studied in \cite{SV}.
It would be an interesting problem to replace the conditions $(*)$ and $({\rm generic})$ by an intrinsic one.

Another important formula is residue formula.
The characteristic property of residue formula is that it describes the cohomology intersection number as a sum of iterated residues at maximal intersections of divisors at infinity, and therefore it is completely an extrinsic formula.
Moreover, residue formula can be formulated for a logarithmic connection in general.
Let $X$ be a smooth complex projective variety of dimension $n$, $D=\cup_{i=1}^ND_i$ be a simple normal crossing divisor, \mychange{$E$} be a vector bundle on $X$ and let $\nabla=\nabla_+:\mathcal{O}_X(\mychange{E})\rightarrow\Omega^1_X(\log D)\otimes \mychange{E}$ be a logarithmic connection.
We write $E^\vee$ (resp. \mychange{$\nabla_-:\mathcal{O}_X(E^\vee)\rightarrow\Omega^1_X(\log D)\otimes E^\vee$}) for the dual bundle (resp. dual connection).
\mychange{If none of eigenvalues of $\Res_{i}(\nabla_+):=\Res_{D_i}(\nabla_+)$ is zero,} the cohomology intersection form (at the middle degree) $\langle\bullet,\bullet\rangle_{ch}:\mathbb{H}^n\left(X;(\Omega^\bullet_X(\log D)\otimes \mychange{E^\vee},\nabla_-)\right)\otimes_{\C} \mathbb{H}^n\left(X;(\Omega^\bullet_X(\log D)\otimes \mychange{E},\nabla_+)\right)\rightarrow\C$ is canonically defined through cup product and trace map.
When $X$ is a divisor completion of $U$, $\mychange{E}$ is the trivial bundle and $\nabla_+$ is given by $d+dF\wedge$, this definition of the cohomology intersection number is compatible with the previous one.
Any logarithmic $n$-form \mychange{$\omega_+\in\Homo^0(X,\Omega^n_X(\log D)\otimes E)$ (resp. $\omega_-\in\Homo^0(X,\Omega^n_X(\log D)\otimes E^\vee)$) naturally defines a cohomology class $[\omega_+]\in\mathbb{H}^n\left(X;(\Omega^\bullet_X(\log D)\otimes {E},\nabla_+)\right)$ (resp. $[\omega_-]\in\mathbb{H}^n\left(X;(\Omega^\bullet_X(\log D)\otimes E^\vee,\nabla_-)\right)$).}
\begin{thm}[Theorem \ref{thm:1}, residue formula]
Assume the condition $(!)_+:$ none of eigenvalues of $\Res_{D_i}(\nabla)$ is a non-negative integer along any component $D_i$.
For any logarithmic $n$-forms $\omega_\pm\in\Homo^0(X,\Omega^n_X(\log D)\otimes E_\pm)$, one has a formula
\begin{equation}\label{eqn:CIN_Intro}
\langle [\omega_-],[\omega_+]\rangle_{ch}=(-2\pi\ii)^n\sum_{P_n}\langle \Res_{P_n}(\omega_+)|\Res_{P_n}(\nabla)^{-1}|\Res_{P_n}(\omega_-)\rangle,
\end{equation}
where the index $P_n=(\mu(1),\dots,\mu(n))$ runs over the set of increasing sequences $1\leq\mu(1)<\cdots<\mu(n)\leq N$ such that $\cap_{i=1}^nD_{\mu(i)}\neq\varnothing$, $\Res_{P_n}(\omega)$ is the iterated residue $\Res_{\mu(m)}\circ\cdots\circ\Res_{\mu(1)}(\omega)$, $\Res_{P_n}(\nabla):=\Res_{\mu(1)}(\nabla)|_{D(P_n)}\circ\cdots\circ\Res_{\mu(n)}(\nabla)|_{D(P_n)}$ and $\langle\bullet|\bullet|\bullet\rangle$ is the duality bracket (cf. \S\ref{sec:2.1}).
\end{thm}

\noindent
The assumption $(!)_+$ can be weakened as in \S\ref{sec:2.1}
As far as the author knows, the proof of this formula is known only for the case when $X$ is a projective space, $D$ is a generic hyperplane arrangement and $E$ is the trivial bundle (\cite{Matsu}), though it is widely believed and used without any proof in the literature (\cite{Goto3},\cite{MM},\cite{MimachiYoshida},\cite{M},\cite{MizeraLong}).
An obstacle for the proof is that the Hodge-to-de Rham spectral sequence does not degenerate at $E_1$-terms.
In \cite{Matsu}, K.Matsumoto used an analytic technique, namely he constructed a suitable family of smooth representatives $\omega_\pm^\ve$ employing the Dolbeault's resolution and computed the cohomology intersection number by taking a limit $\ve\rightarrow 0$.
As the argument in \cite{Matsu} heavily relies on the existence of a global coordinate on a projective space, a major modification of the proof is needed.
\S\ref{sec:33} of this paper is devoted to a proof of Theorem \ref{thm:1}.
An important consequence of the formulas (\ref{eqn:AEIntro}) and (\ref{eqn:CIN_Intro}) is the following result originally due to S.Mizera in a more restricted setting (\cite{M}).
\begin{cor}[Corollary \ref{cor:2.6}]
Assume $\alpha$ is generic. For logarithmic $n$-forms, cohomology intersection number on $\Homo^n(\alpha)$ coincides with Grothendieck residue.
\end{cor}

\noindent
Once the corollary above is formulated and proved precisely, we can relate the stringy integral to the cohomology intersection number.
{\it Positive geometry} consists of a triplet of a variety $X$, its positive part $X_{\geq 0}$ and the top-dimensional logarithmic differential form $\Omega$ called the canonical form with some technical conditions (\cite{AHP}).
Once a positive geometry is given, one considers an integral of a regulator term times the canonical form $\Omega$ over the positive part $X_{\geq 0}$ to derive various physical properties.
The integral is called the stringy integral (\cite{AH}).
In this paper, we assume that, after a suitable re-parametrization of $X_{\geq 0}$, the stringy integral takes the form
\begin{equation}\label{eqn:SI}
\int_{\R^n_{\geq 0}}x_1^{u_1}\cdots x_n^{u_n}q_1(x)^{-v_1}\cdots q_e(x)^{-v_e}\omega_0,
\end{equation}
where $q_i(x)$ are Laurent polynomials with positive coefficients, $u_i,v_j$ are real parameters with $v_j>0$ and $\omega_0=\frac{dx_1}{x_1}\wedge\cdots\wedge\frac{dx_n}{x_n}$ is the canonical form in a new coordinate $x$.
In \cite{StuT}, the integral (\ref{eqn:SI}) is referred to as {\it marginal likelihood integral} of a positive model.
Although the integral (\ref{eqn:SI}) is a transcendental function of $u_i,v_j$ in general, its leading term has a polytopal formula.
Namely, under a suitable assumption on parameters $u_i,v_j$, the limit
\begin{equation}\label{eqn:amp0}
{\rm amplitude}:=\lim_{\ve\rightarrow+0}\ve^n\int_{\R^n_{\geq 0}}(x_1^{u_1}\cdots x_n^{u_n}q_1(x)^{-v_1}\cdots q_e(x)^{-v_e})^\ve\omega_0
\end{equation}
exists and is equal to the normalized volume of a convex polytope (\cite[\S2 Claim 1 and \S4 Claim 2]{AH}).
In this setting, we set $U:=\{ x\in(\C^*)^n\mid q_j(x)\neq 0, \; j=1,\dots,e\}$.
We can regard the canonical form $\omega_0$ as an element of an algebraic de Rham cohomology group associated to the integrand of (\ref{eqn:SI}).
In this paper, we obtain a
\begin{thm}[Theorem \ref{thm:Amplitude}]
For generic parameters $u_1,\dots,u_n,v_1,\dots,v_e$, one has a formula
\begin{equation}\label{eqn:medama}
{\rm amplitude}=\frac{\langle[\omega_0],[\omega_0]\rangle_{ch}}{(-2\pi\ii)^n}.
\end{equation}
\end{thm}

\noindent
In view of the importance of the cohomology intersection number and positive geoemetry in statistics and physics, the formula (\ref{eqn:medama}) illustrates how these notions are closely related.
In \S\ref{sec:4}, we will also provide a formula of the cohomology intersection number in terms of a regular triangulation of a polytope.
This function {\tt mt\_gkz.principal\_normalizing\_constant} is an implementation in risa/asir package {\tt mt\_gkz.rr} (\cite{ICMS}, \cite{MTGKZ}).







\mychange{The structure of the paper is as follows:}
in \S\ref{sec:2}, we provide the precise formulations of our results.
For the sake of clarity, we fist discuss residue formula in \S\ref{sec:2.1} and then discuss stationary phase formula in \S\ref{sec:2.2}.
\S\ref{sec:4} is devoted to the formulation of Theorem \ref{thm:Amplitude} as well as its proof.
We will also deal with the complex stringy integral of \cite[\S8]{AH}.
In \S\ref{sec:33} and \S\ref{sec:3}, we prove Theorem \ref{thm:1} and Theorem \ref{thm:2}.
The author is grateful to Bernd Sturmfels who let him know the paper \cite{StuT}.
He thanks Yoshiaki Goto and Nobuki Takayama for valuable discussions.
\mychange{He also thanks an anonymous referee whose comment improved the exposition of this paper.}
This work is supported by JSPS KAKENHI Grant Number 19K14554 and JST CREST Grant Number JP19209317 including AIP challenge program.





\section{Results}\label{sec:2}
\subsection{Residue formula}\label{sec:2.1}
Let $X$ be a smooth projective variety over $\C$, let $D=\sum_{j=1}^ND_j\subset X$ be a simple normal crossing divisor and let $(E,\nabla)$ be a logarithmic connection with poles along $D$.
The symbol $\Omega^p_{\log}$ denotes the sheaf of logarithmic $p$-forms $\Omega^p_X(\log D)$.
\mychange{We write $\Omega_+^p$ for the sheaf $\Omega^p_{\log}\otimes E$ and write $\Omega_-^p$ for the sheaf $\Omega^p_{\log}\otimes E^\vee$ where $E^\vee$ is the dual bundle of $E$.
We set $\nabla_+=\nabla$ and write $(E^\vee,\nabla_-)$ for the dual connection of $(E,\nabla)$.}
For any integer $j\in\Z$, $\nabla_\pm$ naturally induces a logarithmic connection on $E_\pm(jD)$ which is still denoted by $\nabla_\pm$.
\mychange{For any $j\in\Z$, let the symbol
\begin{equation}
\Homo^p_\pm(jD):=\mathbb{H}^p\left(X;(\Omega_\pm^\bullet(jD),\nabla_{\pm})\right)
\end{equation}
denote logarithmic de Rham cohomology group.
We also set
\begin{equation}
\Homo^p_\pm:=\Homo^p_\pm(0\cdot D).
\end{equation}
}
Wedge product $(\Omega_-^{\bullet},\nabla_-)[n]\otimes_{\C} (\Omega_+^\bullet(-D),\nabla_+)\rightarrow (\Omega_{\log}^{\bullet}(-D),d)[n]$ induces cup product of hypercohomology cohomology groups $\Homo^p_-\otimes\Homo^{2n-p}_+(-D)\rightarrow\mathbb{H}^{2n}\left(X;(\Omega_{\log}^{\bullet}(-D),d)\right)$.
Combining this operation with trace map ${\rm tr}:\mathbb{H}^{2n}\left(X;(\Omega_{\log}^{\bullet}(-D),d)\right)\rightarrow\C$, we obtain a perfect pairing $\langle\bullet,\bullet\rangle_{+}:\Homo^p_-\otimes\Homo^{2n-p}_+(-D)\rightarrow\C$ (see \cite[Proposition 50.20.4]{ST} whose proof can easily be adapted to the current setting).
Let $\mathcal{E}_X^{(p,q)}$ denote the sheaf of $(p,q)$-differential forms on $X^{an}$.
Here, the superscript $an$ stands for analytification.
We set $\EE_{\pm}^{p,q}:=\Omega_{\pm}^p\otimes_{\mathcal{O}^{an}_X}\EE^{(0,q)}_X$ and $\EE_{\pm}^{k}:=\oplus_{p+q=k}\EE_{\pm}^{p,q}$.
Dolbeault's lemma shows that the natural morphism $(\Omega_{\pm}^\bullet(jD),\nabla_\pm)^{an}\rightarrow (\EE_{\pm}^\bullet(jD),\nabla_\pm+\dbar)$ is a quasi-isomorphism.
Taking into account that the natural morphism $\mathbb{H}^{2n}\left(X;(\Omega_{\log}^{\bullet}(-D),d)\right)\rightarrow\Homo^{n}\left(X;\Omega_X^{n}\right)$ is an isomorphism, we see that the pairing $\langle\bullet,\bullet\rangle_{+}$ is given by \mychange{a formula} $\langle\bullet,\bullet\rangle_{+}:\Homo^p_-\otimes_{\C}\Homo^{2n-p}_+(-D)\ni [\omega_-]\otimes[\omega_+]\mapsto \frac{1}{(2\pi\ii)^n}\int_X\omega_-\wedge\omega_+\in\C$ where $\omega_\pm$ are representatives by smooth forms.
Note that \mychange{the value} $\langle[\omega_-],[\omega_+]\rangle_{+}$ belongs to a subfield $k\subset \C$ when $X,$ $D$ and $(E,\nabla)$ are defined over $k$.
In the same manner, one can also define a perfect pairing $\langle\bullet,\bullet\rangle_{-}:\Homo^p_-(-D)\otimes_{\C}\Homo^{2n-p}_+\rightarrow\C$.
Let \mychange{the symbol} $\Res_i(E,\nabla)\in\Homo^0\left( D_i;{\rm End}(E)\right)$ denote the residue of $(E,\nabla)$ along $D_i$ (\cite[2.5]{EV}).
We simply write $\Res_i(\nabla)$ \mychange{instead of} $\Res_i(E,\nabla)$ when there is no fear of confusion.
\mychange{Let us consider the following conditions} 

\begin{equation*}
\text{\underline{$(!)_+$:} none of eigenvalues of $\Res_i(\nabla)$ is in $\Z_{\leq 0}$.}
\end{equation*}
\begin{equation*}
\text{\underline{$(!)_-$:} none of eigenvalues of $\Res_i(\nabla)$ is in $\Z_{\geq 0}$.}
\end{equation*}

\noindent
Under the condition $(!)_\pm$, one has a natural isomorphism $\Homo^p_\pm(-D)\tilde\rightarrow\Homo^p_\pm$ for any $p$ (\cite[2.9]{EV}). When the condition $(!)_\pm$ holds, we write $\reg_\pm:\Homo^p_\pm\tilde\rightarrow\Homo^p_\pm(-D)$ for the corresponding isomorphism. Note that \mychange{the cohomology group is purely $n$-codimensional, i.e. the identity $\Homo^p_\pm=0$ ($p\neq n$) is true if the condition
\begin{equation*}
\text{\underline{(reg):} none of eigenvalues of $\Res_i(E,\nabla)$ is in $\Z$}
\end{equation*}
\noindent
holds.}
\mychange{In the following,} we assume either $(!)_+$ or $(!)_-$.
\mychange{For any element $[\omega_\pm]\in\Homo^n_\pm$, we set $\langle [\omega_-],[\omega_+]\rangle_{ch}:=(2\pi\ii)^n\langle [\omega_-],\reg_+[\omega_+]\rangle_{+}=(2\pi\ii)^n\langle \reg_-[\omega_-],[\omega_+]\rangle_{-}$.}
We call $\langle \bullet,\bullet\rangle_{ch}:\Homo^n_-\otimes_{\C}\Homo^n_+\rightarrow\C$ the cohomology intersection form.\footnote{\mychange{Although the cohomology intersection form $\langle \bullet,\bullet\rangle_{ch}$ depends on a given integrable connection $(E,\nabla)$, we use a notation which does not clarify this dependence as it makes the expression simple and it is customary in the literature (\cite{CM},\cite{Goto3}, \cite{MimachiYoshida}).}}
It is customary to write $\langle\omega_-,\omega_+\rangle$ for $\langle[\omega_-],[\omega_+]\rangle_{ch}$.
When (reg) is true, the cohomology intersection form is compatible with \mychange{topological pairings} as follows:
we set $U:=X\setminus D$.
For the sake of simplicity of the exposition, let us assume that $U$ is affine.
\footnote{We can easily generalize Theorem \ref{thm:1} to the case when $U$ is not affine using hypercohomology group. We do not assume that $D$ is an ample divisor in this article though we can always choose it to be so (\cite[Theorem 1]{Goodman}) if we replace $X$ by a suitable blowing-up. However, this process may violate our assumption on the residues of $(E,\nabla)$ in general.}
Let $\mathcal{L}^\pm$ be the sheaf on $U^{an}$ of flat sections of the analytification $\nabla_\mp^{an}$ and let $\iota:U^{an}\hookrightarrow X^{an}$ be the natural embedding.
The condition (reg) implies that there is a canonical isomorphism $\iota_!\mathcal{L}^\pm\simeq(\Omega_{\mp}^\bullet(jD),\nabla_\mp)^{an}$ for any $j\in\Z$ in the derived category (\cite[II, 3.13,3.14]{Del} and \cite[2.10]{EV}).
In view of GAGA principle (\cite{Serre}), we obtain a canonical isomorphism $\Homo_\pm^n\simeq\Homo^n(\Omega_{U}^\bullet(U),\nabla_\pm)$.
Moreover, the canonical morphism $\iota_!\mathcal{L}^\pm\rightarrow\R\iota_*\mathcal{L}^\pm$ is also an isomorphism (see the proof of \cite[Lemma 3]{CDO} which can be adapted in our setting immediately).
In view of Poincar\'e duality, this yields the comparison isomorphism $\Homo_n(U^{an};\mathcal{L}^\pm)\overset{\sim}{\rightarrow}\Homo_n^{\rm lf}(U^{an};\mathcal{L}^\pm)$ where the superscript {\rm lf} stands for the locally finite (or Borel-Moore) homology group.
Let $\langle\bullet,\bullet\rangle_{per}:\Homo_n(U^{an};\mathcal{L}^\pm)\otimes_{\C}\Homo^n_\pm\ni [\Gamma]\otimes[\omega]\mapsto\int_{\Gamma}\omega\in\C$ denote the period pairing and let $\langle\bullet,\bullet\rangle_{h}:\Homo_n(U^{an};\mathcal{L}^-)\otimes_{\C}\Homo_n(U^{an};\mathcal{L}^+)\rightarrow\C$ be the homology intersection pairing.
Fixing a basis $\{[\Gamma_i^\pm]\}_{i=1}^r\subset\Homo_n(U^{an};\mathcal{L}^\pm)$, we write $(h^{-1}_{ij})_{i,j=1}^r$ for the inverse matrix of a matrix $\left( \langle[\Gamma_i^-],[\Gamma_j^+]\rangle_h\right)_{i,j=1}^r$. For any $[\omega_\pm]\in\Homo_\pm^n$, \mychange{the following identity is true (\cite{CM}, \cite{FSY}):
\begin{equation}\label{eqn:TPR1}
\langle\omega_-,\omega_+\rangle_{ch}=\sum_{i,j=1}^r\left(\int_{\Gamma_j^+}\omega_+\right)h^{-1}_{ji}\left(\int_{\Gamma_i^-}\omega_-\right).
\end{equation}
}
We call the identity (\ref{eqn:TPR1}) twisted period relation.
Note that the number $(2\pi\ii)^n$ does not appear in the formula above.

We also introduce a standard notation $\langle\omega_+,\omega_-\rangle_{ch}:=(-1)^n\langle\omega_-,\omega_+\rangle_{ch}$ for $[\omega_\pm]\in\Homo^n_\pm$.
For any section $\omega$ of $\Omega_{\pm}^p$, we define the symbol $\Res_1(\omega)$ as follows: let $x=(x_1,\dots,x_n)$ be a local coordinate so that \mychange{the equality} $D_1=\{ x_1=0\}$ \mychange{holds locally and} $\omega$ is locally expressed as $\omega=\frac{dx_1}{x_1}\wedge\omega_1+\omega_2$ where $\omega_1$ and $\omega_2$ do not have poles along $D_1$.
We set $\Res_1(\omega):=\rest_1(\omega_1)$ where $\rest_1$ denotes the restriction to $D_1$.
We can also define the symbol $\Res_i(\omega)$ for any $i=1,\dots,N$.
The following commutativity relation is useful.
\begin{prop}
For any section $\omega$ of $\Omega^p(\log (D\setminus D_i))\otimes E$, one has a relation
\begin{equation}
\Res_i(\nabla\omega)=\Res_i(\nabla)\rest_i(\omega).
\end{equation}
\end{prop}

For any multi-index $P_m=\{ \mu(1),\dots,\mu(m)\}\subset\{ 1,\dots,N\}$ of cardinality $m$ ($\leq n$), we set $D(P_m)=\cap_{i=1}^mD_{\mu(i)}$ and $\Res_{P_m}(\nabla)=\Res_{\mu(1)}(\nabla)|_{D({P_m})}\circ\cdots\circ\Res_{\mu(m)}(\nabla)|_{D({P_m})}\in\Homo^0\left( D(P_m);{\rm End}(E)\right)$. Since $\nabla$ is integrable, we have a commutativity relation $[\Res_i(\nabla)|_{D(P_m)},\Res_j(\nabla)|_{D(P_m)}]=0$ where the bracket is commutator product.
Therefore, the definition of $\Res_{P_m}(\nabla)$ does not depend on the choice of the order of $\mu(1),\dots,\mu(m)$.
On the other hand, we fix the order as $\mu(1)<\cdots<\mu(m)$ and set $\Res_{P_m}(\omega):=\Res_{\mu(m)}\circ\cdots\circ\Res_{\mu(1)}\omega$ for any section $\omega$ of $\Omega^p_{\pm}$. Any element $\omega$ of $\Homo^0(X;\Omega^n_{\pm})$ defines a cohomology class $[\omega]$ of $\Homo^n_+$ or $\Homo^n_-$ through the natural morphism $\Homo^0\left(X;\Omega_{\pm}^n\right)\rightarrow\Homo^n_{\pm}$. For any $\omega_\pm\in \Homo^0\left(X;\Omega_{\pm}^n\right)$, we write $\langle \Res_{P_n}(\omega_+)|\Res_{P_n}(\nabla)^{-1}|\Res_{P_n}(\omega_+)\rangle$ for $\sum_{z\in D(P_n)}\langle \Res_{P_n}(\nabla)^{-1}_z\Res_{P_n}(\omega_+)_z,\Res_{P_n}(\omega_-)_z\rangle$ where $\langle\bullet,\bullet\rangle$ denotes the duality bracket $\langle\bullet,\bullet\rangle:\mychange{E_{z}\otimes E_{z}^\vee}\rightarrow\C$.
The following theorem is a general form of the results \cite{CM} and \cite{Matsu}.

\begin{thm}\label{thm:1}
Assume either $(!)_+$ or $(!)_-$. For any $\omega_\pm\in\Homo^0\left(X;\Omega_{\pm}^n\right)$, one has
\begin{equation}\label{eqn:CIN}
\langle [\omega_+],[\omega_-]\rangle_{ch}=(2\pi\ii)^n\sum_{P_n;\ D({P_n})\neq\varnothing}\langle \Res_{P_n}(\omega_+)|\Res_{P_n}(\nabla)^{-1}|\Res_{P_n}(\omega_-)\rangle.
\end{equation}

\end{thm}

\begin{example}\label{exa:0}
{Let $L$ be an ample line bundle over $X$.
We take a linear system $s_0,\dots,s_N\in\Homo^0(X;L)$ so that $D=\bigcup_{i=0}^N\{ s_i=0\}$ is simple normal crossing and defines a projective embedding $X\to\PP^N$.
We consider a finite dimensional complex vector space $V$ and its endomorphisms $\alpha_i\in {\rm End}_\C(V)$ such that $\sum_{i}\alpha_H=0$.
We write $\alpha_i^+$ for $\alpha_i$ and $\alpha^-_i$ for the transpose of $\alpha_i$.
Let $E$ be the trivial vector bundle $X\times V$ on $X$.
We set $\omega_\pm:=\sum_{i=0}^N\alpha_i^\pm d\log s_i$ and $\nabla_\pm:=d\pm\omega_\pm\wedge:\Omega^0_\pm\rightarrow\Omega_\pm^1$ and assume that it defines an integrable connection, i.e., $\omega_\pm\wedge\omega_\pm=0$.
Here, the precise meaning of $\omega_\pm$ is $\omega_\pm:=\sum_{i=1}^N\alpha_i^\pm d\log (s_i/s_0)\in\Omega_{\log}^1\otimes {\rm End}(E_{\pm})$.
The setting above appears in the work of Riemann-Wirtinger integrals where  $X$ is an abelian variety and $D$ is a theta divisor(\cite{MW},\cite{W}).
}
\end{example}

\begin{example}\label{exa:1}
We discuss a special case of Example \ref{exa:0}.
Let $\mathscr{A}\subset \PP^n$ be a projective hyperplane arrangement.
In the following, we use the standard terminologies of hyperplane arrangements (see e.g., \cite[Chapter 3]{OT}).
We consider a finite dimensional complex vector space $V$ and its endomorphisms $\alpha_H\in {\rm End}_\C(V)$ ($H\in\mathscr{A}$) such that $\sum_{H\in\mathscr{A}}\alpha_H=0$.
Let $E$ be the trivial vector bundle $\PP^n\times V$ on $\PP^n$.
We write $l_H$ for the defining equation of the hyperplane $H\in \mathscr{A}$.
We set $\omega_\pm:=\sum_{H\in\mathscr{A}}\alpha_H^\pm d\log l_H\in\Omega_{\log}^1\otimes {\rm End}(E_{\pm})$, $\nabla_\pm:=d\pm\omega_\pm\wedge:\Omega^0_\pm\rightarrow\Omega_\pm^1$ and assume that it defines an integrable connection, i.e., $\omega_\pm\wedge\omega_\pm=0$.

Let us construct a standard compactification $X$ of $\PP^n\setminus\mathscr{A}$.
We write $D(\mathscr{A})$ for \mychange{the set} of dense edges.
The compactification $\pi:X\rightarrow\PP^n$ is constructed as a composition of blowing-ups along dense edges in an increasing order of their dimensions.
It can be shown that $D:=\pi^{-1}\mathscr{A}\subset X$ is a simple normal crossing divisor (\cite[10.8]{Var2}, see also \cite[Theorem 4.2.4]{OT}).
Moreover, each dense edge $e$ corresponds to a unique irreducible component $D_e$ of $D$ and we have an irreducible decomposition $D=\cup_{e\in D(\mathscr{A})}D_e$.
When $e$ is in $\mathscr{A}$, the corresponding component $D_e$ is a proper transform of $e$, otherwise $D_e$ is the exceptional divisor generated by the blowing-up along $e$.
By abuse of notation, we write $\nabla_\pm$ for $\pi^*\nabla_\pm$ and $E$ for $\pi^*E$.
For any $e\in D(\mathscr{A})$, we set $\mathscr{A}_e:=\{ H\in\mathscr{A}\mid e\subset H\}$.
The assumption $(!)_+$ in this setting reads 
\begin{equation*}
\text{\underline{$(!)_+$:} For any dense edge $e\in D(\mathscr{A})$, none of eigenvalues of $\sum_{H\in \mathscr{A}_e}\alpha_H$ is in $\Z_{\leq 0}$.}
\end{equation*}

\noindent
We pick a hyperplane in $\mathscr{A}$ and name it $H_\infty$.
Let $A^n_\pm\subset\Homo^0(X;\Omega_\pm^n)$ be the spanning subspace of logarithmic forms \mychange{$\omega(H_1,\dots,H_n):=\pi^*(d\log (l_{H_1}/l_{H_\infty})\wedge\cdots\wedge d\log (l_{H_n}/l_{H_\infty}))$ with $\{ H_1,\dots,H_n\}\subset\mathscr{A}\setminus\{ H_\infty\}$.}
\mychange{The result of} \cite{ESV} (see also \cite{STV}) shows that the canonical morphism $A^n\rightarrow\Homo^n_\pm$ is surjective if the condition $(!)_\pm$ is satisfied. For any $H\in\mathscr{A}\setminus\{ H_\infty\}$ and $e\in D(\mathscr{A})$, we have a formula
\begin{equation}
\mychange{\Res_{H_e}(\pi^*d\log (l_H/l_{H_\infty}))}=
\begin{cases}
1&(H\in\mathscr{A}_e\change{\text{ and }H_\infty\notin\mathscr{A}_e})\\
-1&(H_\infty\in\mathscr{A}_e\change{\text{ and }H\notin\mathscr{A}_e})\\
0&(\change{otherwise})
\end{cases}.
\end{equation}
On the other hand, for any tuple $P=(H_{e_1},\dots,H_{e_n})$ such that $\cap_{i=1}^nH_{e_i}$ is a single point, we have \mychange{$\Res_P(\omega(H_1,\dots,H_n))=\det(\Res_{H_{e_i}}\pi^*d\log (l_{H_j}/l_{H_\infty}))_{i,j=1}^n$.}
The cohomology intersection form $\langle\bullet,\bullet\rangle_{ch}:\Homo_-^n\otimes\Homo^n_+\rightarrow\C$ is determined by these formulae and (\ref{eqn:CIN}).
In particular, when we have a perfect knowledge of the tuples $P=(H_{e_1},\dots,H_{e_n})$ such that $\cap_{i=1}^nH_{e_i}$ is a single point, we can determine the cohomology intersection number in an exact way.
The simplest example is a generic arrangement.
In this case, theorem \ref{thm:1} reproduces the main result of \cite{Matsu}.
Another typical example of such an arrangement is \mychange{Coxeter} arrangement (\cite{AD}).
The formula of cohomology intersection numbers of {\it Parke-Taylor forms} in \cite[Theorem 2.1]{MizeraLong} is indeed a simple adaptation of the discussion above because the configuration space $\mathcal{M}_{0,n}$ is nothing but the complement of a Coxeter arrangement of type A and the compactification $X$ is the Deligne-Knudsen-Mumford compactification $\overline{\mathcal{M}}_{0,n}$ (\cite{Kap}).

\end{example}

\subsection{Stationary phase formula}\label{sec:2.2}
In this section, we consider the case when the vector bundle $E$ is the trivial bundle and a tuple of nowhere vanishing regular functions  $f_1,\dots,f_m$ on the complement $U=X\setminus D$ is given.
We regard $f_i$ as a morphism $f_i:U\rightarrow\C^*$.
Let ${\rm ord}_{D_i}f_j$ denote the vanishing order of $f_j$ along $D_i$.
We assume $U$ is affine and consider the following condition:
\begin{equation*}
\underline{(*):}\text{  \mychange{f}or any }i=1,\dots,N\text{, there is at least one $j$ such that }{\rm ord}_{D_i}f_j\neq 0.
\end{equation*}
For a tuple of complex numbers $\alpha=(\alpha_1,\dots,\alpha_m)$, we set $F_\alpha:=\sum_{i=1}^m\alpha_i\log f_i$ and ${\rm Crit}(F_\alpha):=\{ x\in U\mid dF(x)=0\}$. In view of the argument \cite[\S2.3,2.4]{Huh}, the condition $(*)$ implies that Varchenko's conjecture is true, namely, for a generic choice of $\alpha$\mychange{$\in\C^m$},\footnote{\mychange{Although it is assumed that $\alpha\in\Z^m$ in \cite[\S2.3,2.4]{Huh}, the result of [loc. cit.] holds true for $\alpha\in\C^m$. See also \cite{Sil}.}} one has
\begin{enumerate}
\item ${\rm Crit}(F_\alpha)$ is a 0-dimensional variety;
\item For any $p\in {\rm Crit}(F_\alpha)$, $F_\alpha$ is non-degenerate at $p$; and
\item The cardinality of ${\rm Crit}(F_\alpha)$ is equal to the signed Euler characteristic $(-1)^{n}\chi(U)$.
\end{enumerate}
The condition $(*)$ also implies that the condition (reg) is true for a generic parameter $\alpha$ with respect to a logarithmic connection $\nabla=d+dF_\alpha\wedge:\mathcal{O}_X\rightarrow\Omega_{\log}^1$.
Indeed, if we set $l_i(s):=\sum_{j=1}^m({\rm ord}_{D_i}f_j) s_j$, (reg) is equivalent to the condition $l_i(\alpha)\notin\Z$.
We identify the dual connection $\nabla_-$ with $d-dF_\alpha\wedge:\mathcal{O}_X\rightarrow\Omega_{\log}^1$.
Under the condition (reg), we have a canonical isomorphism $\mathbb{H}^n\left(X;(\Omega^\bullet_X(*D),\nabla_\pm)\right)\simeq\Homo^n_\pm$ in view of the quasi-isomorphism $(\Omega_{\log}^\bullet,\nabla)\simeq (\Omega^\bullet_X(*D),\nabla)$ (\cite{EV}). Therefore we have a canonical morphism $\Homo^0\left(X;\Omega^n_X(*D)\right)\rightarrow\Homo^n_\pm$ which is surjective as $U$ is affine.
A typical example that the condition $(*)$ is satisfied is a smooth very affine variety, i.e., a smooth closed subvariety $U$ of $\T={\rm Spec}\;\C[t_1^{\pm1},\dots,t_m^{\pm1}]$.
The function $f_j$ is the restriction of the coordinate function $t_j:\T\rightarrow\C^*$ to $U$. In \cite[\S2.3, LEMMA 2.3]{Huh}, a compactification $X$ of $U$ was constructed so that it satisfies the condition $(*)$.

Let us consider a family version of this construction.
We fix a parameter $\alpha\in\C^m$ so that the condition
\begin{equation*}
\text{\underline{(generic):} }\ \ \ \ l_i(\alpha)\neq 0\text{ for any }i=1,\dots,N
\end{equation*}
is true.
We introduce a \mychange{scalling} parameter $\tau$ and set $A:=\C[\tau]$.
We write $A_{\rm loc}\supset A$ for the ring obtained by inverting the polynomials $l_i(\alpha)\tau\pm j$ for $j\in\Z$ and write $K$ for the field of convergent Laurent series at infinity $\C[\tau]\{\tau^{-1}\}$.
For any $\C$-algebra $S$, we write $X_S$ for the base change (as a scheme) with respect to the structure morphism $\C\rightarrow S$.
Note that we have identifications $\Homo^k(X;\mathfrak{F}\otimes_{\C}S)\simeq\Homo^k(X;\mathfrak{F})\otimes_{\C}S\simeq\Homo^k(X_S;\mathfrak{F}\otimes_{\C}S)$ for any $k\in\Z$ and for any quasi-coherent sheaf $\mathfrak{F}$ on $X$.
We set $\nabla_{\pm}:=d\pm\tau dF\wedge:\Omega^p_{\log}\otimes_{\C}S\rightarrow\Omega^{p+1}_{\log}\otimes_{\C}S$ for $S=A,A_{\rm loc},K$.
We define the trace map $\mathbb{H}^{2n}(X;(\Omega^\bullet_{\log}(-D)\otimes_{\C} S,d))\simeq \mathbb{H}^{2n}(X;(\Omega^\bullet_{\log}(-D),d))\otimes_{\C}S\rightarrow S$ as a base change of the trace map $\mathbb{H}^{2n}(X;(\Omega^\bullet_{\log}(-D),d))\rightarrow\C$.
We set $\Homo^p_\pm(jD)[\tau]:=\mathbb{H}^p\left(X;(\Omega_{\log}^\bullet(jD)\otimes_{\C}A,\nabla_{\pm})\right)$ for any $j\in\Z$.
We define the cohomology intersection form $\langle\bullet,\bullet\rangle_{ch}:\Homo_-^p[\tau]\otimes_{A}\Homo^{2n-p}_+(-D)[\tau]\rightarrow A$ as a composition of the cup product $\Homo_-^p[\tau]\otimes_{A}\Homo^{2n-p}_+(-D)[\tau]\rightarrow\mathbb{H}^{2n}(X;(\Omega^\bullet_{\log}(-D)\otimes_{\C}A,d))$ and the trace morphism $\mathbb{H}^{2n}(X;(\Omega^\bullet_{\log}(-D)\otimes_{\C}A,d))\rightarrow A$ multiplied by $(2\pi\ii)^n$.
For any $\tau_0\in\C$, we have a natural specialization morphism ${\rm sp}(\tau_0):\Homo^p_\pm(jD)[\tau]\rightarrow\Homo^p_\pm(jD)$ defined by means of a substitution $\tau=\tau_0$.
They fit into a commutative diagram 
\begin{equation}
\xymatrix{
\Homo^p_{\pm}(jD)[\tau]\ar[r] \ar[d]^-{{\rm sp}(\tau_0)} &\Homo^p_\pm((j+1)D)[\tau]\ar[d]^-{{\rm sp}(\tau_0)}\\
\Homo^p_{\pm}(jD)\ar[r] &\Homo^p_\pm((j+1)D).
}
\end{equation}
We set $\Homo_\pm^p(jD)[\tau]_{\rm loc}:=\Homo_\pm^p(jD)[\tau]\otimes_{A}A_{\rm loc}$ and $\Homo_\pm^p(jD)[\tau]\{\tau^{-1}\}:=\Homo_\pm^p(jD)[\tau]\otimes_{A}K$.
For any $\tau_0\in\C$ such that $\tau_0l_i(\alpha)\notin\Z$, we have a specialization morphism ${\rm sp}(\tau_0):\Homo_\pm^p(jD)[\tau]_{\rm loc}\rightarrow\Homo_\pm^p(jD)$.
We write $\langle\bullet,\bullet\rangle_{ch}:\Homo^p_{-}[\tau]_{\rm loc}\otimes_{A_{\rm loc}}\Homo^{2n-p}_{+}(-D)[\tau]_{\rm loc}\rightarrow A_{\rm loc}$ and $\langle\bullet,\bullet\rangle_{ch}^{\tau}:\Homo^p_{-}[\tau]\{\tau^{-1}\}\otimes_K\Homo^{2n-p}_{+}(-D)[\tau]\{\tau^{-1}\}\rightarrow K$ for the base change of $\langle\bullet,\bullet\rangle_{ch}$.
In view of the identification $\Homo^{p}_{+}(jD)[\tau]\{\tau^{-1}\}\simeq \mathbb{H}^p(X;(\Omega_{\log}^\bullet(jD)\otimes_{\C}K,\nabla_+))\simeq\mathbb{H}^p(X_K;(\Omega_{\log}^\bullet(jD)\otimes_{\C}K,\nabla_+))$, the proof of \cite[Proposition 50.20.4]{ST} shows that $\langle\bullet,\bullet\rangle_{ch}^{\tau}$ is a perfect pairing.
We have the following commutative diagram:
\begin{equation}\label{diag:2.5}
\xymatrix{
\Homo^p_{-}\otimes_{\C}\Homo^{2n-p}_{+}(-D)\ar[r]^-{\langle\bullet,\bullet\rangle_{ch}}&\C\\
\Homo^p_{-}[\tau]_{\rm loc}\otimes_{A_{\rm loc}}\Homo^{2n-p}_{+}(-D)[\tau]_{\rm loc}\ar[r]^-{\langle\bullet,\bullet\rangle_{ch}} \ar[u]^-{{\rm sp}(\tau_0)} \ar[d]&A_{\rm loc}\ar[u]_-{{\rm sp}(\tau_0)}\ar[d]\\
\Homo^p_{-}[\tau]\{\tau^{-1}\}\otimes_{K}\Homo^{2n-p}_{+}(-D)[\tau]\{\tau^{-1}\}\ar[r]^-{\langle\bullet,\bullet\rangle_{ch}^\tau} &K.
}
\end{equation}
As in the proof of \cite[2.9]{EV}, we see that the canonical morphism $(\Omega^\bullet_{\pm}(jD)\otimes_{\C}S,\nabla_\pm)\rightarrow (\Omega^\bullet_{\pm}((j+1)D)\otimes_{\C}S,\nabla_\pm)$ is a quasi-isomorphism for $S=A_{\rm loc},K$ and for any $j\in\Z$.
Note that $\Res_i(\mathcal{O}_X(jD),\nabla)=l_i(\alpha)\tau\mychange{-}j$ is invertible both in $A_{\rm loc}$ and $K$.
In the following, we use the identification $\Homo^p_{\pm}[\tau]\otimes_{A}S\simeq\mathbb{H}^p(X;(\Omega_X^\bullet(*D)\otimes_{\C}S,\nabla_\pm))=\mathbb{H}^p(U;(\Omega_U^\bullet\otimes_{\C}S,\nabla_\pm))$ for $S=A_{\rm loc},K$.
The last term vanishes when $p>n$ and the perfectness of $\langle\bullet,\bullet\rangle_{ch}^\tau$ shows the \mychange{identity} $\Homo^p_\pm[\tau]\{\tau^{-1}\}=0$ unless $p=n$.
Since $U$ is affine, there is a natural surjective morphism $\Homo^0(U;\Omega^n_{U}\otimes_{\C}S)\simeq \Homo^0(U;\Omega_U^n)\otimes_{\C}S\rightarrow\Homo^n_\pm\otimes_{A}S$.
In particular, $\Homo^n_\pm[\tau]\otimes_{A}S$ is generated by $\Homo^0(U;\Omega_U^n)$ over $S$.
For any elements $\omega_\pm\in\Homo^0(U;\Omega_U^n)$, we write $[\omega_\pm]$ for \mychange{its image} through the natural morphism $\Homo^0(U;\Omega_U^n)\rightarrow\Homo^n_\pm\otimes_{A}S$.
In view of the commutative diagram (\ref{diag:2.5}), for any $\omega_\pm\in\Homo^0(U;\Omega_U^n)$, the cohomology intersection number $\langle[\omega_-],[\omega_+]\rangle_{ch}^\tau$ belongs to $A_{\rm loc}$ and therefore, reconstructed from values ${\rm sp}(\tau_0)(\langle[\omega_-],[\omega_+]\rangle_{ch})$ for generic $\tau_0$.
By abuse of notation, we will write $\langle\omega_-,\omega_+\rangle_{ch}^\tau$ for $\langle[\omega_-],[\omega_+]\rangle_{ch}^\tau$.

Let us establish a formula for the leading term of the pairing $\langle\bullet,\bullet\rangle_{ch}^\tau$.
For any local coordinate $x=(x_1,\dots,x_n)$ of $U^{an}$, we write $H_{F_\alpha,x}$ for the Hessian determinant $\det (\frac{\partial^2F_\alpha}{\partial x_i\partial x_j})_{i,j=1}^n$.
Any holomorphic $n$-form $\omega$ on $U^{an}$ can be written in the form $\omega=a(x)dx_1\wedge\cdots\wedge dx_n$ locally. We set $\frac{\omega}{dx}:=a(x)$.

\begin{thm}\label{thm:2}
Assume the conditions $(*)$ and {\rm (generic)}. For any $\omega_\pm\in\Homo^0(U;\Omega^n_U)$, one has a Laurent expansion
\begin{equation}\label{eqn:exp}
\frac{\langle\omega_-,\omega_+\rangle_{ch}^{\tau}}{(2\pi\ii)^n}=\tau^{-n}\sum_{k\geq 0}K^{(k)}(\omega_-,\omega_+)\tau^{-k}\ \ (|\tau|\rightarrow+\infty)
\end{equation}
of which the leading term is
\begin{equation}\label{eqn:2.7}
K^{(0)}(\omega_-,\omega_+)=\frac{1}{(2\pi\ii)^n}\sum_{p\in {\rm Crit}(F_\alpha)}\int_{\Gamma_p}\left[\frac{\frac{\omega_+}{dx}\frac{\omega_-}{dx}}{\partial_{x_1}F_\alpha\cdots\partial_{x_n} F_\alpha}\right]dx
\end{equation}
where $x=(x_1,\dots,x_n)$ is a local coordinate on a neighborhood $U_p$ of $p$, $dx=dx_1\wedge\cdots \wedge dx_n$ and $\Gamma_p$ is a cycle given by $\Gamma_p=\{x\in U_p\mid |\partial_{x_1}F_\alpha(x)|=\ve,\dots,|\partial_{x_n}F_\alpha(x)|=\ve\}$ for some small number $\ve>0$ of which the  orientation is given by $d\arg (\partial_{x_1}F_\alpha)\wedge\cdots\wedge d\arg (\partial_{x_n}F_\alpha)>0$.
\end{thm}

\begin{rem}
\leavevmode\newline
\begin{enumerate}
\item
(\ref{eqn:2.7}) is known as Grothendieck residue (cf. \cite[Chapter 5]{GH}).
In view of the {\it transformation law} of Grothendieck residue (p657 of loc. cit.), we see that each term of the sum (\ref{eqn:2.7}) is independent of the choice of the local coordinate $x$.
\item
When $\alpha$ is generic so that the Varchenko's conjecture is true, the leading term of (\ref{eqn:exp}) is given by the stationary phase formula
\begin{equation}\label{eqn:leading}
K^{(0)}(\omega_-,\omega_+)=\sum_{p\in{\rm Crit}(F_\alpha)}\frac{1}{H_{F_\alpha,x}(p)}\frac{\omega_+}{dx}(p)\frac{\omega_-}{dx}(p).
\end{equation}
\item
Let us consider regular functions $f_1,\dots,f_m:(\C^*)^n={\rm Spec}\;\C[x_1^\pm,\dots,x_n^\pm]\rightarrow\C$.
If $U=\{ x\in(\C^*)^n\mid f_i(x)\neq0,\ i=1,\dots,m\}$, (\ref{eqn:leading}) can be written as 
\begin{equation}
K^{(0)}(\omega_-,\omega_+)=\sum_{p\in{\rm Crit}(F_\alpha)}\frac{1}{H_{F_\alpha,x}^{\rm toric}(p)}\frac{\omega_+}{\omega_0}(p)\frac{\omega_-}{\omega_0}(p),
\end{equation}
where $\omega_0=\frac{dx_1}{x_1}\wedge\cdots\wedge\frac{dx_n}{x_n}$ and $H_{F_\alpha,x}^{\rm toric}(p)$ is the toric Hessian $\det (\theta_i\theta_jF(p))_{i,j}$.
Here, $\theta_i$ denotes the $i$-th Euler operator $\theta_i=x_i\frac{\partial}{\partial x_i}$.
\item
When $X,D,f_1,\dots,f_m$ are defined over a subfield $k\subset\C$, \mychange{both} $\frac{\langle\omega_-,\omega_+\rangle_{ch}^{\tau}}{(2\pi\ii)^n}$ and each term $K^{(k)}(\omega_-,\omega_+)$ belongs to $k(\alpha)$. 
\item
The formula (\ref{eqn:2.7}) is valid when ${\rm Crit}(F_\alpha)$ is discrete and the sum $\sum_{p\in{\rm Crit}(F_\alpha)}\mu_p$ is equal to $(-1)^n\chi(U^{an})$.
Here, $\mu_p$ is the Milnor number of $F_\alpha$ at $p$ defined by $\mu_p:=\dim_{\C}\mathcal{O}^{an}_{U,p}/\langle \partial_{x_1}F_\alpha,\dots,\partial_{x_n} F_\alpha\rangle$.
This condition is valid if the parameter $\alpha$ is chosen so that the condition (generic) is true.
This is because {\it logarithmic Gau\ss-Bonnet theorem} is true in this setting (\cite[Remark 1.4 and \S4.2]{Sil} see also \cite[\S2.3,2.4]{Huh}).
Note that $\{ x\in X\mid dF(x)=0\}\cap U={\rm Crit}(F_\alpha)$.
\end{enumerate}
\end{rem}

We can recover the result of S. Mizera in its full generality (\cite{M}).

\begin{cor}\label{cor:2.6}
Assume the condition $(*)$. For any $\alpha\in\C^m$ such that $l_i(\alpha)\notin\Z$ ($i=1,\dots,N$) and $\omega_\pm\in\Homo^0(X;\Omega_{\log}^n)$, one has
\begin{equation}\label{eqn:2.6}
\frac{\langle\omega_-,\omega_+\rangle_{ch}}{(2\pi\ii)^n}=K^{(0)}(\omega_-,\omega_+).
\end{equation}
\end{cor}

\begin{proof}
Theorem \ref{thm:1} shows that $\langle\omega_-,\omega_+\rangle_{ch}^\tau$ is homogeneous of order $-n$ with respect to $\tau$. This implies that higher pairings $K^{(k)}(\omega_-,\omega_+)$ ($k\geq 1$) all vanish.
\end{proof}

\begin{example}
We inherit the notation of Example \ref{exa:1}.
Let us consider the special case when $V$ is a one-dimensional complex vector space.
We may regard each $\alpha_H$ as a scalar.
The condition $(*)$ is true for the compactification $X$ and the condition $(generic)$ is precisely the condition 
\begin{equation*}
\text{\underline{$(generic)$:} For any dense edge $e\in D(\mathscr{A})$, none of $\sum_{H\in \mathscr{A}_e}\alpha_H$ is in $\Z$.}
\end{equation*}
Now the result \cite{ESV} (see also \cite{STV}) shows that the natural morphism $(A^\bullet,\omega_\pm\wedge)\rightarrow (\Omega_U^\bullet(U),\nabla_\pm)$ is a quasi-isomorphism.
Corollary \ref{cor:2.6} shows that the following diagram is commutative, which was first discovered in \cite{M}:
\begin{equation}
\xymatrix@C10em{
\Homo^n_-\otimes\Homo^n_+\ar[r]^-{\langle\bullet,\bullet\rangle_{ch}}&\C\\
\Homo^n(A^\bullet,\omega_-\wedge)\otimes\Homo^n(A^\bullet,\omega_+\wedge)\ar[u]^-{\simeq}\ar[r]^-{K^{(0)}(\bullet,\bullet)}&\C.\ar[u]_-{=}
}
\end{equation}
\end{example}



\noindent
Another corollary of Theorem \ref{thm:2} is the following criterion of finding a basis of the cohomology group. We write $\Homo_\pm^n(\alpha)$ for $\Homo^n_\pm$ to emphasize its dependence on $\alpha$.

\begin{cor}\label{cor:2.7}
Assume the conditions $(*)$ and ${\rm (generic)}$. For any $\C$-basis $[\omega_1],\dots,[\omega_r]\in\Homo^n\left(U;(\Omega^\bullet_U,dF\wedge)\right)$, the corresponding cohomology classes $\{[\omega_i]\}_i\subset\Homo_\pm^n(\tau\alpha)$ form a basis if $\tau$ is chosen generically. 
\end{cor}

\begin{proof}
The leading term $K^{(0)}(\bullet,\bullet):\Homo^n\left(U;(\Omega^\bullet_U,dF\wedge)\right)\otimes_{\C}\Homo^n\left(U;(\Omega^\bullet_U,dF\wedge)\right)\rightarrow\C$ defines a perfect pairing (Grothendieck duality, \cite[p659, p707]{GH}).
This proves that the determinant $\det \left( K^{(0)}(\omega_i,\omega_j)\right)_{i,j=1}^r$ is non-vanishing.
We write $[\omega^\pm_i]$ for $[\omega_i]\in\Homo_\pm^n[\tau]\{\tau^{-1}\}$ to emphasize the signature.
Then, theorem \ref{thm:2} shows that the leading term of the determinant $\det \left( \langle\omega_i^-,\omega_j^+\rangle_{ch}^{\tau}\right)_{i,j=1}^r$ is $\tau^{-nr}\det \left( K^{(0)}(\omega_i,\omega_j)\right)_{i,j=1}^r\neq 0$, hence the assertion is true in view of the commutative diagram (\ref{diag:2.5}).
\end{proof}

We conclude this section with a simple observation on the difference structure of algebraic de Rham cohomology groups.
Let us assume $\alpha\in\Z^m$.
$\Homo^{n}_\pm[\tau]\{\tau^{-1}\}$ is naturally equipped with a left action of a ring of difference operators.
We use the identification $\Homo^{n}_\pm[\tau]\{\tau^{-1}\}\simeq\mathbb{H}^n\left( U;(\Omega^\bullet_U\otimes_{\C}K,\nabla_{\pm})\right)$.
Let $R$ be a non-commutative $\C$-algebra generated by elements $\tau,\s$ with a relation $\s\tau=(\tau+1)\s$.
Note that $A$ and $K$ admit an action of the difference operator $\s:\tau\mapsto\tau+1$ so that $\s(\tau^{-m})=\tau^{-m}\sum_{k=0}^\infty(-1)^k\frac{(m)_k}{k!}\tau^{-k}$.
Here, we set $(m)_k:=m(m+1)\cdots (m+k-1)$.
We set $R\{\tau^{-1}\}:=K\otimes_{A}R$ which is again a $\C$-algebra.
An action of $\s$ on $\Omega^n_{U}\otimes_{\C}K$ is given by $\s\bullet_\pm(a(\tau)\omega)=f^{\pm\alpha}a(\tau+1)\omega$ wher $a(\tau)\in K$ and $\omega$ is a local section of $\Omega_{U}^n$.
It is easy to see the relations $\s\bullet_+(\nabla_+\omega(\tau))=\nabla_+(\s\bullet_+\omega(\tau))$ and $\s\bullet_-(\nabla_-\omega(\tau))=\nabla_-(\s\bullet_-\omega(\tau))$ hold for any local section $\omega(\tau)\in \Omega_{U}^n\otimes_{\C}K$.
For any $R\{\tau^{-1}\}$-modules $M_1,M_2$, we can define an $R\{\tau^{-1}\}$-module structure on $M_1\otimes_{K}M_2$ by $\s(m_1\otimes m_2)=\s m_1\otimes\s m_2$.
By the definition of $\langle\bullet,\bullet\rangle_{ch}^\tau$, it is straightforward to prove the following
\begin{prop}
\mychange{Suppose that $\alpha\in\Z^m$.}
Then, $\langle\bullet,\bullet\rangle_{ch}^\tau:\Homo^n_{-}[\tau]\{\tau^{-1}\}\otimes_{K}\Homo^{n}_{+}[\tau]\{\tau^{-1}\}\rightarrow K$ is a perfect pairing of $R\{\tau^{-1}\}$-modules.
\end{prop}

\begin{rem}
For any $\omega_\pm\in\Homo^0(U;\Omega^n_U)$, we have a compatibility relation
\begin{equation}
K^{(k)}(f^{-\alpha}\omega_-,f^{\alpha}\omega_+)=\sum_{l=0}^kK^{(l)}(\omega_-,\omega_+)\frac{(-1)^{k-l}(n+l)_{k-l}}{(k-l)!}.
\end{equation}
\end{rem}

\subsection{Amplitude is a cohomology intersection number}\label{sec:4}

In this subsection, we relate results of \cite{AH} and \cite{StuT} to cohomology intersection numbers.
We consider Laurent polynomials $q_j$ ($j=1,\dots,e$) in variables $x=(x_1,\dots,x_n)$ with positive coefficients.
Let ${\rm New}(q_j)$ denote the convex hull of the exponent vectors of monomials in $q_j$.
We assume that the Minkowski sum $\sum_{j=1}^e{\rm New}(q_j)$ is an $n$-dimensional polytope.
Let us consider real parameters $u_i$ and $v_j\change{>0}$ and set
\begin{equation}\label{eqn:amp}
{\rm amplitude}:=\lim_{\ve\rightarrow+0}\ve^n\int_{\R^n_{\geq 0}}(x_1^{u_1}\cdots x_n^{u_n}q_1(x)^{-v_1}\cdots q_e(x)^{-v_e})^\ve\omega_0.
\end{equation}
Here, we set $\omega_0=\frac{dx_1}{x_1}\wedge\cdots\wedge\frac{dx_n}{x_n}$ .

We provide a combinatorial formula for the amplitude in terms of regular triangulations.
Let us recall basic notation from \cite{GMH}.
We write $A_j$ ($j=1,\dots,e$) for the matrix obtained by aligning column vectors that appear as exponents of monomials in $q_j$.
We consider the Cayley configuration with $e$-partitions:
\begin{equation}\label{CayleyConfigu}
A
=
\left(
\begin{array}{ccc|ccc|c|ccc}
1&\cdots&1&0&\cdots&0&\cdots&0&\cdots&0\\
\hline
0&\cdots&0&1&\cdots&1&\cdots&0&\cdots&0\\
\hline
 &\vdots& & &\vdots& &\ddots& &\vdots& \\
\hline
0&\cdots&0&0&\cdots&0&\cdots&1&\cdots&1\\
\hline
 &A_1& & &A_2& &\cdots & &A_e& 
\end{array}
\right).
\end{equation}
Let $L_A$ be its kernel lattice, i.e., the kernel of the left multiplication $A\times:\Z^{N}\rightarrow\Z^{n+e}$ where $N$ is the number of columns of $A$.
By a row transformation, we may assume that the first column vector of $A_j$ is a zero vector.
We write $A_j=({\bf 0}|\tilde{A}_j)$ and write $\Q \tilde{A}_j$ for the $\Q$-vector subspace of $\Q^n$ spanned by the column vectors of $\tilde{A}_j$.
By our assumption on $q_j$, we have $\sum_{j=1}^e\Q \tilde{A}_j=\Q^{n}$.
In this case, we can find an $n\times n$ $\Z$-matrix $Q$ such that $\sum_{j=1}^e\Z \tilde{A}_j=\Z Q$.
We can apply the change of variables $y=x^{Q^{-1}}$ to obtain an integral with the associated Cayley configuration $\tilde{A}$ in which we replace $A_j$ by $({\bf 0}|Q^{-1}\tilde{A}_j)$.
This configuration satisfies $\Z\tilde{A}=\Z^{n+e}$.
In order to avoid cumbersome notation, we simply assume $\Z A=\Z^{n+e}$. We set $\delta:={}^t(v_1,\dots,v_e,u_1,\dots,u_n)$.
For any vector $v\in\R^N$ such that $Av=-\delta$, we set 
\begin{equation}
\varphi_v(z):=\sum_{w\in L_A}\frac{z^{w+v}}{\prod_{i=1}^N\Gamma(1+w_i+v_i)}. \end{equation}
Here, $w_i$ and $v_i$ denote $i$-th entries of $w$ and $v$.
A subset $\s\subset\{ 1,\dots,N\}$ is called a simplex if $|\s|=n+e$ and $\det A_\s\neq 0$.
Here, $A_\s$ is a square matrix obtained by aligning column vectors of $A$ labeled by the index set $\s$.
For any simplex $\s$ we set $\bs:=\{ 1,\dots,N\}\setminus\s$. If ${\bf k}=(k_i)_{i\in\bs}\in\Z^{\bs}$, the equation $Av=-\delta$ has a unique solution $v=v_{\s}^{\bf k}$ such that $v_i=k_i$ ($i\in\bs$).
The $i$($\in\s$)-th entry of $v_\s^{\bf k}$ is denoted by $p_{\s i}(-\delta)$.
We set $\varphi_{\s,{\bf k}}(z;\delta):=\varphi_{v_\s^{\bf k}}(z)$.
For any generic vector $\omega\in\R^N$, we write $T(\omega)$ for the corresponding regular triangulation and regard it as a collection of simplices $\s$ (\cite[Chapter 8]{Stu}).
For a fixed regular triangulation $T$, we say that the parameter vector $\delta$ is {very generic} if $A_\s^{-1}(\delta+{\bf k})$ has no integral entry for any  simplex $\sigma\in T$ and any ${\bf k}\in\Z^{n+e}$.
We set 
\begin{equation}
U_\sigma=\left\{z\in(\C^*)^N\mid {\rm abs}\left(z_\sigma^{-A_\sigma^{-1}{\bf a}(j)}z_{j}\right)<R, \text{for all } j\in \bs\right\},
\end{equation}
where $R>0$ is a small positive real number and abs stands for the absolute value. For a regular triangulation $T$, we set $U_T:=\cap_{\s\in T}U_\s\neq\varnothing$.

\noindent
We set $P:=\sum_{j=1}^ev_j{\rm New}(q_j)$ \mychange{and} $U:=\{ x\in(\C^*)^n\mid q_j(x)\neq 0\ (j=1,\dots,e)\}$ and identify it with a closed subvariety $\{ (x;t)\in(\C^*)^{n+e}\mid t_{1}=q_1(x),\dots,t_e=q_e(x)\}$ of a complex torus $(\C^*)^{n+e}={\rm Spec}\, \C[x_1^{\pm1},\dots,x_{n}^{\pm1},t_1^{\pm1},\dots,t_e^{\pm1}]$.
In the notation of \S\ref{sec:2.2}, $f_1,\dots,f_{n+e}$ are simply the coordinate projections $x_i,t_j:(\C^{n+e})^*\rightarrow\C^*$ restricted to $U$ and it satisfies the condition $(*)$ for a suitable projective compactification $X$ of $U$ (\cite[Lemma 2.3]{Huh}).
We set $L:=\sum_{i=1}^nu_i\log x_i-\sum_{j=1}^ev_j\log q_j(x)$.
We define a logarithmic integrable connection on the trivial bundle on $X$ by $\nabla_\pm:=d\pm dL\wedge:\mathcal{O}_X\rightarrow\Omega^1_{\log}$.
We also write $\omega_0^\pm$ for $\omega_0$ to emphasize that the cohomology class $[\omega_0^\pm]$ is taken in the cohomology group $\Homo^n_\pm$.
For any convex polytope $\Delta$ in an Euclidean space $\R^n$ such that the origin is contained in $\Delta$, we set $\Delta^\circ:=\{ \phi\in\Hom_{\R}(\R^N,\R)\mid \langle\phi,v\rangle\geq -1 \text{ for all }v\in \Delta\}$ and $\vol_{\Z}(\Delta):=n!{\rm vol}(\Delta)$ where the symbol vol stands for the Lebesgue measure of $\R^n$.

\begin{thm}\label{thm:Amplitude}
The integral (\ref{eqn:amp}) is convergent if $u$ is in the interior of $P$. For a generic choice of the parameter $\delta$, it is given by
\begin{align}
{\rm amplitude}
&=\vol_\Z(P-u)^\circ\label{eqn:1st}\\
&=\sum_{p\in{\rm Crit}(L)}\frac{\change{(-1)^n}}{\det H_{L}^{\rm toric}(p)}\label{eqn:2nd}\\
&=\frac{\change{\langle \omega_0^+,\omega_0^-\rangle_{ch}}}{(2\pi\ii)^n}\label{eqn:3rd}
\end{align}
Furthermore, the following identity is true:
\begin{equation}
\frac{\change{\langle \omega_0^+,\omega_0^-\rangle_{ch}}}{(2\pi\ii)^n}
=\change{v_1\cdots v_e}\sum_{\s\in T}\frac{1}{|\det A_\s|}\prod_{i\in\s}\frac{1}{p_{\s i}(\change{\delta})}.\label{eqn:Amplitude}
\end{equation}
\end{thm}

\begin{proof}
The convergence of the integral (\ref{eqn:amp}) and the equalities (\ref{eqn:1st}) and (\ref{eqn:2nd}) are \mychange{stated in} \cite[Theorem 13]{StuT} which summarizes the result of \cite[\S2 Claim 1, \S4 Claim 2 and \S7]{AH}.
\change{Note the the Newton polytope map $\Phi:=-(x_1\frac{\partial}{\partial x_1}\log L,\dots,x_n\frac{\partial}{\partial x_n}\log L):\R^n_{>0}\tilde{\rightarrow} {\rm Int}(P-u)$ is an orientation preserving diffeomorphism. In \cite[Theorem 13]{StuT}, signature $(-1)^n$ is missing.}
The third equality (\ref{eqn:3rd}) follows from Theorem \ref{thm:2}.
Let us show the equality (\ref{eqn:Amplitude}) under some additional assumption to illustrate its relation to our previous work \cite{GMH}.
A complete proof is \mychange{provided} in the appendix.
If we write $q_j(x)=\sum_{k=1}^{N_j}z_k^{(j)}x^{{\bf a}^{(j)}(k)}$ with $z_k^{(j)}>0$, we take the variable $z=(z_k^{(j)})$ so that $z\in U_T$.
Then, \cite[Theorem 1.1]{GMH} in our setting reads 
\begin{align}
&\frac{\langle \omega_0^+,\omega^-_0\rangle_{ch}}{(2\pi\ii)^n}\nonumber\\
=&v_1\cdots v_e\sum_{\s\in T}\sum_{[A_{\bs}{\bf k}]\in\Z^{n+e}/\Z A_\s}\frac{1}{|\det A_\s|}\frac{\pi^{n+e}}{\sin\pi A_\s^{-1}(\delta+A_{\bs}{\bf k})}
\varphi_{\s,{\bf k}}( z;\delta)
\varphi_{\s,{\bf k}^\prime}( z;-\delta).\label{eqn:GMH}
\end{align}
Here, ${\bf k}^\prime$ denotes the element such that $[-A_{\bs}{\bf k}]=
[A_{\bs}{\bf k}^\prime]$ holds in the group $\Z^{n+e}/\Z A_\s$.
Now, we consider a toric compactification $X$ of $U$ constructed in \cite{Hov} (see also \cite[\S3]{MH}).
For any generic $z$, the residue of $\nabla_\pm$ along any irreducible component of the boundary divisor $D=X\setminus U$ is an integral linear combination of entries of $\delta$.
In particular, we can conclude from Theorem \ref{thm:1} that $\langle \omega_0^+,\omega^-_0\rangle_{ch}$ does not depend on $z$.
Let us take a generic vector $\omega\in\R^N$ and let us write $\mathcal{T}$ for the set of associated top-dimensional standard pairs in the sense of \cite[\S 3.2]{SST}.
Each element of $\mathcal{T}$ takes a form $(\s,{\bf k})$ where $\s\in T(\omega)$ and ${\bf k}\in\Z_{\geq 0}^{\bs}$.
Moreover, for a fixed $\s\in T(\omega)$, the set $\{ A_{\bs}{\bf k}\mid (\s,{\bf k})\in\mathcal{T}\}$ gives a complete system of representatives of $\Z^{n+e}/\Z A_{\s}$.
We choose $(\s,{\bf k}),(\s,{\bf k}^\prime)\in\mathcal{T}$ so that $[-A_{\bs}{\bf k}]=[A_{\bs}{\bf k}^\prime]$ in $\Z^{n+e}/\Z A_\s$.
In view of \cite[Theorem 3.4.2]{SST}, the $\omega$-initial monomial of the product $\varphi_{\s,{\bf k}}( z;\delta)\varphi_{\s,{\bf k}^\prime}( z;-\delta)$ is $z^{v_\s^{\bf k}+v_\s^{{\bf k}^\prime}}$ whose $\omega$-degree is given by
\begin{equation}
\langle\omega,v_\s^{\bf k}+v_\s^{{\bf k}^\prime}\rangle=
\begin{cases}
=0&({\bf k}={\bf 0})\\
>0&({\bf k}\neq{\bf 0})
\end{cases}.
\end{equation}
As the cohomology intersection number $\langle \omega_0^+,\omega^-_0\rangle_{ch}$ does not depend on $z$, it is equal to the $\omega$-initial term of the right-hand side of (\ref{eqn:GMH}), which proves the last equality (\ref{eqn:Amplitude}).
\end{proof}


\change{
At this stage, one may notice that Theorem \ref{thm:Amplitude} also proves a formula 
\begin{equation}\label{eqn:CT}
\vol_\Z(P-u)^\circ=v_1\cdots v_e\sum_{\s\in T}\frac{1}{|\det A_\s|}\prod_{i\in\s}\frac{1}{p_{\s i}(\delta)},
\end{equation}
which is of purely combinatorial nature.
In the appendix, we prove (\ref{eqn:CT}) in a purely combinatorial way.
This provides a complete proof of the identity (\ref{eqn:Amplitude}).
}


Let us also discuss complex stringy integrals (\cite[\S8]{AH}).
For any vector $v\in\C^N$ such that $Av=-\delta$, we set 
\begin{equation}
\psi_v(z):=\sum_{w\in L_A}\frac{z^{u}}{\prod_{i=1}^N\Gamma(1+w_i+v_i)}.
\end{equation}
For any simplex $\s$ and ${\bf k}\in\Z^{\bs}$, we set $\psi_{\s,{\bf k}}(z;\delta)=\psi_{v_\s^{\bf k}}(z)$. \cite[Theorem 6.2]{MHA} combined with \cite[Theorem 2.4 and 2.5]{GMH} proves the following
\begin{thm}\label{thm:Absolute}
For any regular triangulation $T$ and $z\in U_T$, one has an identity
\begin{align}
&\frac{1}{(-2\pi\ii)^n}\int_{\C^n}\prod_{i=1}^n|x_i|^{2u_i}\prod_{j=1}^e|q_j(x)|^{-2v_j}\omega_0\wedge\overline{\omega_0}\nonumber\\
=&\frac{\Gamma({\bf 1}-v)}{\Gamma(v)}\sum_{\s\in T}\sum_{[A_{\bs}{\bf k}]\in\Z^{(n+k)\times 1}/\Z A_\s}\frac{(-1)^{|{\bf k}|}}{|\det A_\s|}\frac{\pi^{n+k}}{\sin\pi A_\s^{-1}(\delta+A_{\bs}{\bf k})}
|z_\s|^{-2A_\s^{-1}\delta}
\psi_{\s,{\bf k}}\left( z;\delta\right)
\psi_{\s,{\bf k}}\left( \bar{z};\delta\right).\label{AbsoluteQR}
\end{align}
Note that both sides of (\ref{AbsoluteQR}) are meromorphic functions in $\delta$ and $\psi_{\s,{\bf k}}\left( z;\delta\right)
\psi_{\s,{\bf k}}\left( \bar{z};\delta\right)=|\psi_{\s,{\bf k}}\left( z;\delta\right)|^2$ when $\delta$ is real.
\end{thm}

\noindent
If we introduce a scaling parameter $\ve$ and replace $\delta$ by $\ve\delta$, it is straightforward to compute the term-wise scaling limit ${\ve\rightarrow+0}$ in (\ref{AbsoluteQR}) and it recovers the result of \cite[\S8]{AH}.
\begin{cor}
\begin{align}
&\frac{1}{(-2\pi\ii)^n}\lim_{\ve\rightarrow 0}\ve^n\int_{\C^n}\prod_{j=1}^e|q_j(x)|^{-2\ve v_j}\prod_{i=1}^n|x_i|^{2\ve u_i}\frac{dx}{x}\wedge\frac{d\bar{x}}{\bar{x}}\nonumber\\
=&\vol_\Z(P-u)^\circ\\
=&\change{v_1\cdots v_e}\sum_{\s\in T}\frac{1}{|\det A_\s|}\prod_{i\in\s}\frac{1}{p_{\s i}(\change{\delta})}\\
=&\frac{\langle \change{\omega_0^+,\omega_0^-}\rangle_{ch}}{(2\pi\ii)^n}.
\end{align}
\end{cor}

\section{Proof of Theorem \ref{thm:1}}\label{sec:33}
We use the notation of \S\ref{sec:2.1} \mychange{and} prove Theorem \ref{thm:1} only when the condition $(!)_+$ is true.
The case of the condition $(!)_-$ can be proved in the same manner.
In this section, we always equip $X$ with the analytic topology.
In order to simplify the notation, we will write $X$ for $X^{an}$ throughout this section.
We begin with some basic formulas of currents.
Let \mychange{the symbol} $\DD^{\prime (p,q)}_X$ denote the sheaf of $(p,q)$-currents on $X^{an}$.
For any local section $T$ of $\DD^{\prime (p,q)}_X$ and a locally defined smooth $(n-p,n-q)$-form $\varphi$ on $X$, we write $\langle T,\varphi\rangle\in\C$ for the natural duality pairing of currents. 
Since any element of $\EE_{\rm log}^{p,q}:=\Omega_{\log}^p\otimes_{\mathcal{O}_{X^{an}}}\EE^{(0,q)}_X$ is a differential form with locally integrable coefficients, we have a natural morphism $\PV:\EE_{\rm log}^{p,q}\hookrightarrow\DD^{\prime (p,q)}_X$.
We write $\iota_j:D_j\hookrightarrow X$ for the natural embedding.

\begin{lem}\label{lem:3.1}
Let $f,\varphi$ be smooth functions on $\C$ and assume that $\varphi$ has a compact support. Then the following identity is true:

\begin{equation}
\int_{\C}f(z)\frac{\partial\varphi}{\partial\bar{z}}\frac{dz}{z}\wedge d\bar{z}=
\int_{\C}\bar{\partial}f\wedge\varphi(z)\frac{dz}{z}+2\pi\ii f(0)\varphi(0)
\end{equation}
\end{lem}

\noindent
The proof of the lemma is elementary. Now we recall that $\C^n$ equipped with a holomorphic coordinate $(z_1,\dots,z_n)$ is oriented so that $(\frac{\ii}{2})^ndz_1\wedge\cdots dz_n\wedge d\bar{z}_1\wedge\cdots d\bar{z}_n=d\Re z_1\wedge\cdots d\Re z_n\wedge d\Im z_1\wedge \cdots\wedge d\Im z_n>0$.
The residue morphism $\Res_j:\Omega_{\log}^p\rightarrow\Omega_{D_j}^{p-1}(\log(D\setminus D_j))$ naturally induces a morphism $\EE_{\log}^{(p,q)}=\Omega_{\log}^p\otimes_{\mathcal{O}_{X^{an}}}\EE^{(0,q)}_X\rightarrow\EE_{D_j}^{(p-1,q)}(\log(D\setminus D_j))$ which is still denoted by $\Res_j$.

\begin{prop}\label{prop:ResF}
For any $\omega\in\Homo^0(X;\EE^{(n,n-1)}_{\rm log})$, one has an identity
\begin{equation}\label{eqn:ResCur}
\bar{\partial}\PV(\omega)=\PV(\bar{\partial}\omega)+(-1)^{n-1}2\pi\ii\sum_{j=1}^N\iota_{j*}\PV(\Res_j(\omega)),
\end{equation}
where $\iota_{j*}$ denotes the direct image of a current.
\end{prop}

\begin{proof}
By means of a partition of unity, we only need to prove (\ref{eqn:ResCur}) in a coordinate neighborhood.
Let us take a coordinate neighborhood $(V,z_1,\dots,z_n)$ on $X$ so that $V\cap D=\{ z_1\cdots z_s=0\}$.
We set $\delta_i:=d\log z_i$ ($i=1,\dots,s$) and $\delta_i:=dz_i$ ($i=s+1,\dots,n$). Let us set $d\bar{z}_{\hat{j}}:=d\bar{z}_{1}\wedge\cdots\wedge\widehat{d\bar{z}_{j}}\wedge\cdots\wedge d\bar{z}_{n}$, $\delta:=\delta_1\wedge\cdots\wedge\delta_n$ and $\omega:=f\delta\wedge d\bar{z}_{\hat{j}}$.
We assume $1\leq j\leq s$.
Note that $\Res_j(\omega)=(-1)^{j-1}\rest_{z_j=0}f\delta_{\hat{j}}\wedge d\bar{z}_{\hat{j}}$.
For any test function $\varphi$ with its support contained in $V$, we have
\begin{align}
\langle \dbar \PV(\omega),\varphi\rangle&=\langle \PV(\omega),\dbar\varphi\rangle\\
&=(-1)^{n-j}\int_{\C^n}f(z)\frac{\partial\varphi}{\partial\bar{z}_j}\delta\wedge d\bar{z}\\
&=(-1)^{n-j}\prod_{i=1,i\neq j}^n(\int_{\C}\delta_i\wedge d\bar{z}_i)\int_{\C}f(z)\frac{\partial\varphi}{\partial\bar{z}_j}\delta_j\wedge d\bar{z}_j\\
&\overset{Lemma \ref{lem:3.1}}{=}(-1)^{n-j}\int_{\C^n}\frac{\partial f}{\partial\bar{z}_j}\varphi(z)\delta\wedge d\bar{z}+(-1)^{n-j}2\pi\ii\prod_{i=1,i\neq j}^n(\int_{\C}\delta_i\wedge d\bar{z}_i)\rest_{z_j=0}(f\varphi)\\
&=\int_{\C^n}\dbar{\omega}\wedge\varphi+(-1)^{n-1}2\pi\ii\int_{D_j}\Res_j(\omega)\wedge\rest_{z_j=0}\varphi\\
&=\langle \PV(\bar{\partial}\omega)+(-1)^{n-1}2\pi\ii\sum_{j=1}^N\iota_{j*}\PV(\Res_j(\omega)),\varphi\rangle.
\end{align}
\end{proof}

\begin{cor}\label{cor:2.5}
For any $\omega\in\Homo^0(X;\EE^{(n,n-1)}_{\rm log})$ one has an identity
\begin{equation}
\int_X\dbar\omega=(-1)^n2\pi\ii\sum_{j=1}^N\int_{D_j}\Res_j(\omega).
\end{equation}
\end{cor}

\begin{rem}
When $n=1$, this is the classical residue theorem.
\end{rem}

For any local section $\omega$ of $\EE^{(p,q)}_{\log}$, $\PV(\omega)$ is a current of order $0$, i.e., on each open subset it belongs to the topological dual of the space of continuous forms with compact support.
On a compact manifold, the notion of ``almost everywhere'' is well-defined by means of a finite covering.
The following proposition is an easy consequence of Lebesgue's dominance convergence theorem.
\begin{prop}\label{prop:3.5}
Let $\omega\in\Homo^0(X;\EE^{(n,n)}_{\rm log})$ and $\beta_n$ be continuous functions on $X$.
If there exists a positive constant $M>0$ such that $\sup_{x\in X}|\beta_n(x)|<M$ for any $n\in \Z_{\geq 0}$, and $\beta(z)=\lim_{n\rightarrow\infty}\beta_n(z)$ exists almost everywhere on $X$, one has $\lim_{n\rightarrow\infty}\int_X\beta_n\omega=\int_X\beta\omega$.
\end{prop}

\noindent
In the following, all the multi-indices $P_m=(\mu(1),\dots,\mu(m))$ are aligned so that $\mu(1)<\cdots<\mu(m)$.
As in \S\ref{sec:2.1}, we write $\EE^{(p,q)}_{+}$ (resp. $\EE^{(p,q)}_{-}$) for $\EE^{(p,q)}_{\log}\otimes E$ (resp. $\EE^{(p,q)}_{\log}\otimes E^\vee$) and write $\EE^k_\pm$ for $\oplus_{p+q=k}\EE^{(p,q)}_{\pm}$.
For a pair of multi-indices $P_m\subset P_{m+1}=(\mu(0),\dots,\mu(m))$, we set $\delta(P_m;P_{m+1})=(-1)^l$ if $P_{m+1}\setminus P_m=\{ \mu(l)\}$. 

\begin{lem}\label{lem:Solution}
Let $\omega\in\Homo^0\left( X;\Omega_+^n\right)$.
Assume that the condition $(!)_+$ holds.
For each multi-index $P_m$, one can find a neighborhood $V(P_m)$ of $D(P_m)$ and $\psi^{P_m}\in\Homo^0\left( V(P_m);\EE^{n-m}(\log(D\setminus D_{\mu(m)}))\right)$ such that 
\begin{enumerate}
\item $V(P_m)=X$ and $\psi^{P_m}=0$ if $D(P_m)=\varnothing$.
\item $D(P_m)\subset V(P_m)\subset \bigcap_{P_{m-1}\subset P_m}V(P_{m-1})$ if $D(P_m)\neq\varnothing$.
\item \begin{equation}\label{eqn:2.5}(-1)^{m-1}\nabla\psi^{P_m}=\sum_{P_{m-1}\subset P_m}\delta(P_{m-1};P_m)\psi^{P_{m-1}}\text{ on $V(P_m)$ if $D(P_m)\neq\varnothing$.}\end{equation}
\item $\psi^{\varnothing}=\omega$ and $\delta(\varnothing;P_1)=1$.
\item For any $i=0,\dots,m-2$, $\rest_{\mu(m-i)}\circ\cdots\circ\rest_{\mu(m)}\psi^{P_m}$ exists and it belongs to \newline
$\Homo^0\left( D_{\mu(m-i),\dots,\mu(m)};\EE^{n-m}(\log(D\setminus D_{\mu(m-i-1),\dots,\mu(m)}))\right)$.
\item $\rest_{P_m}\psi^{P_m}:=\rest_{\mu(1)}\circ\cdots\circ\rest_{\mu(m)}\psi^{P_m}=\Res_{P_m}(\nabla)^{-1}\Res_{P_m}\omega$.
\end{enumerate}
 
\end{lem}

\begin{proof}
Assume that $\psi^{P_l}$ and $V(P_l)$ with desired properties are constructed for multi-indices $P_l$ of cardinality $|P_l|\leq m-1$. 
Let us pick a multi-index $P_m$ of cardinality $m$.
If $D(P_m)=\varnothing$, we set $V(P_m)=X$ and $\psi^{P_m}=0$.
Suppose $D(P_m)\neq\varnothing$.
We observe that $\tilde{\psi}:=\sum_{P_{m-1}\subset P_m}\delta(P_{m-1};P_m)\psi^{P_{m-1}}$ is $\nabla$-closed.
This is clear when $m=1$.
Let us suppose $m\geq 2$ and fix a multi-index $P_m=(\mu(1),\dots,\mu(m))$.
By induction, we obtain $\nabla\tilde{\psi}=\sum_{P_{m-2}\subset P_{m-1}\subset P_m}\delta(P_{m-2};P_{m-1})\delta(P_{m-1};P_m)\psi^{P_{m-2}}$.
For any $P_{m-2}\subset P_m$, we set $P_{m}\setminus P_{m-2}=\{ \mu(l_1),\mu(l_2)\}$ and $P^i_{m-1}:=P_m\setminus\{ \mu(l_i)\}$ ($i=1,2$).
It is easy to observe that $\delta(P_{m-2};P^1_{m-1})\delta(P^1_{m-1};P_m)=-\delta(P_{m-2};P^2_{m-1})\delta(P^2_{m-1};P_m)$.
This implies that $\tilde{\psi}$ is $\nabla$-closed.
Let us consider a complex 
\begin{equation}
C^\nu:0\rightarrow\EE_{+}^0(-D_j)\overset{\nabla}{\rightarrow} \EE_{+}^1(-D_j){\rightarrow}\cdots{\rightarrow}\EE_{+}^{\nu-1}(-D_j)\overset{\nabla}{\rightarrow}\EE^\nu_+(\log(D\setminus D_j))\overset{\nabla}{\rightarrow}\EE_{+}^{\nu+1}\overset{\nabla}{\rightarrow} \EE_{+}^{\nu+2}\rightarrow\cdots
\end{equation}
for $\nu=0,\dots,n$.
Under the assumption of Lemma \ref{lem:Solution}, $C^\nu$ is quasi-isomorphic to $\iota_!\mathcal{L}^-$ (\cite[II, 3.13,3.14]{Del} and \cite[2.10]{EV}).
If $j:D(P_m)\hookrightarrow X$ denotes the closed embedding, we have $j^{-1}\iota_!\mathcal{L}^-=0$.
Therefore, one has
\begin{equation}\label{eqn:TubeVan}
0=\Homo^k(D(P_m);\iota_!\mathcal{L}^-)=\varinjlim_{D(P_m)\subset V}\Homo^k(V;\iota_!\mathcal{L}^-)=\varinjlim_{D(P_m)\subset V}\mathbb{H}^k(V;C^\nu)
\end{equation}
(cf. \cite[Remark 2.6.9]{KS}).
Since $\tilde{\psi}$ is $\nabla$-closed, it defines an element $[\tilde{\psi}]$ of $\mathbb{H}^{n-m+1}(V;C^{n-m})$ where $V=\cap_{P_{m-1}\subset P_m}V(P_{m-1})$.
Therefore, on a suitably small open neighborhood $V(P_m)$ of $D(P_m)$, we have $[\tilde{\psi}|_{V(P_m)}]=0$ in $\mathbb{H}^{n-m+1}(V(P_m);C^{n-m})$.
Since $C^{n-m}$ is a soft complex, we can find an element $\psi^{P_m}\in\Homo^0\left( V(P_m);\EE^{n-m}(\log(D\setminus D_{\mu(m)}))\right)$ so that $(-1)^{m-1}\nabla\psi^{P_m}=\tilde{\psi}$. 

We apply $\Res_{\mu(m)}$ to the equality $(-1)^{m-1}\nabla\psi^{P_m}=\tilde{\psi}$ to obtain $(-1)^{m-1}\Res_{\mu(m)}(\nabla)\rest_{\mu(m)}(\psi^{P_m})=\delta(\{ \mu(1),\dots,\mu(m-1)\};P_m)\Res_{\mu(m)}(\psi^{\mu(1),\dots,\mu(m-1)})$.
By the induction hypothesis and the relation $\rest_i\circ\Res_{\mu(m)}=\Res_{\mu(m)}\circ\rest_i$, one can show that $\psi^{P_m}$ satisfies the desired property and $\rest_{P_m}\psi^{P_m}=\Res_{\mu(m)}(\nabla)^{-1}\Res_{\mu(m)}(\rest_{\mu(1),\dots,\mu(m-1)}\psi^{\mu(1),\dots,\mu(m-1)})=\Res_{P_m}(\nabla)^{-1}\Res_{P_m}\omega$.

\end{proof}

Now we mimic the construction of \cite{Matsu}.
Let $h_j$ be a smooth function on $X$ which takes the value $1$ on a small neighborhood of $D_j$ and $0\leq h_j\leq 1$.
We also assume that we can take a neighborhood $W_j$ of the support of $h_j$ so that $W(P_m):=\cap_{j\in P_m}W_j\subset V(P_m)$ if $D(P_m)\neq\varnothing$ and $W(P_m)=\varnothing$ if $D(P_m)=\varnothing$.
Therefore, $h_j$ is a pudding function on a neighborhood of $D_j$. We set $g^{\mu,\nu}:=\prod_{\lambda=\mu}^\nu(1-h_{\lambda})$ for $1\leq \mu\leq\nu\leq N$ and set $g^{\mu,\mu-1}:=1$.
Note that we have $1-\sum_{\lambda=\mu}^\nu g^{\mu,\lambda-1}h_{\lambda}=g^{\mu,\nu}$. For any multi-index $P_m=(\mu(1),\dots,\mu(m))$, we set 
\begin{align}
\eta^{\lambda}(P_m):=&g^{\mu(\lambda-1)+1,\mu(\lambda)-1}dh_{\mu(\lambda)}\\
H(P_m):=&\eta^1(P_m)\wedge\cdots\wedge\eta^{m-1}(P_m)g^{\mu(m-1)+1,\mu(m)-1}h_{\mu(m)}\\
G(P_m):=&\eta^1(P_m)\wedge\cdots\wedge\eta^{m}(P_m)
\end{align}
where we set $\mu(0):=0$. We set $\Psi^m:=\sum_{\substack{P_m\subset\{ 1,\dots,N\}\\ |P_m|=m}}H(P_m)\wedge\psi^{P_m}$.
The proof of the following lemma is analogous to that of \cite[Lemma 6.4]{Matsu} and therefore, is omitted.
\begin{lem}
Let $\{\psi^{P_m}\}_{P_m}$ be functions constructed in Lemma \ref{lem:Solution}. Then, one has the following relations:
\begin{enumerate}
\item $\nabla\Psi^1=(1-g^{1,N})\omega+\sum_{P_2}H(P_2)\wedge\nabla\psi^{P_2}+\sum_{P_1=(\mu(1))}G(P_1)g^{\mu(1)+1,N}\wedge\psi^{P_1}$
\item $(-1)^{m-1}\nabla\Psi^m=\sum_{P_m}H(P_m)\wedge\nabla\psi^{P_m}+\sum_{P_{m+1}}H(P_{m+1})\wedge\nabla\psi^{P_{m+1}}+\sum_{P_m}G(P_m)g^{\mu(m)+1,N}\wedge\psi^{P_m}$ ($2\leq m\leq n-1$)
\item $(-1)^{n-1}\nabla\Psi^n=\sum_{P_n}H(P_n)\wedge\nabla\psi^{P_n}+\sum_{P_n}G(P_n)g^{\mu(n)+1,N}\wedge\psi^{P_n}$.
\end{enumerate}
\noindent
In particular, 
\begin{equation}
\omega-\nabla(\sum_{m=1}^n\Psi^m)=g^{1,N}\omega+\sum_{m=1}^n\sum_{P_m}(-1)^mG(P_m)g^{\mu(m)+1,N}\wedge\psi^{P_m}
\end{equation}
vanishes on a neighborhood of $D$.
\end{lem}

Let us prove Theorem \ref{thm:1}.
We apply Lemma \ref{lem:Solution} to $\omega=\omega_+$ and construct a family $\{\psi^{P_m}\}_{P_m}$.
We write $\lim_{h_j\rightarrow{\bf 1}_{D_j}}$ for the limit $W_j\rightarrow D_j$, i.e., it is a limit that the function $h_j$ converges to the characteristic function ${\bf 1}_{D_j}$ of $D_j$ satisfying $0\leq h_j\leq 1$. By definition, we have \mychange{an identity} $\langle\omega+,\omega_-\rangle_{ch}=\langle\omega_+-\nabla(\sum_{m=1}^n\Psi^m),\omega_-\rangle_{ch}$.
The left-hand side does not have the auxiliary function $h_j$.
Therefore, if the iterative limit $\lim_{h_1\rightarrow{\bf 1}_{D_1}}\circ\cdots\circ\lim_{h_N\rightarrow{\bf 1}_{D_N}}\langle\omega_+-\nabla(\sum_{m=1}^n\Psi^m),\omega_-\rangle_{ch}$ exists, it gives $\langle\omega_+,\omega_-\rangle_{ch}$.
We set $g^{P_m}:=g^{\mu(0)+1,\mu(1)-1}\cdots g^{\mu(m-1)+1,\mu(m)-1}$. First of all, we have 
\begin{align}
\langle\omega_+,\omega_-\rangle_{ch}&=\int_X\left(\omega_+-\nabla(\sum_{m=1}^n\Psi^m)\right)\wedge\omega_-\\
&=-\sum_{m=1}^n\int_X\dbar(\Psi^m\wedge\omega_-)\\
&=-\sum_{m=1}^n\sum_{P_m}\int_X\dbar(H(P_m)\wedge\psi^{P_m}\wedge\omega_-)
\end{align}
Hereafter, we fix a multi-index $P_m=(\mu(1),\dots,\mu(m))$ and compute the iterative limit $\lim_{h_1\rightarrow{\bf 1}_{D_1}}\circ\cdots\circ\lim_{h_N\rightarrow{\bf 1}_{D_N}}(-1)\int_X\dbar(H(P_m)\wedge\psi^{P_m}\wedge\omega_-)$. Since $H(P_m)\equiv 0$ if $D(P_m)=\varnothing$, we consider the case $D(P_m)\neq\varnothing$. By Corollary \ref{cor:2.5}, we have \mychange{an equality}

\begin{align}
&-\int_X\dbar(H(P_m)\wedge\psi^{P_m}\wedge\omega_-)\nonumber\\
=&(-1)^{n-1}2\pi\ii\sum_{j=1}^N\int_{D_j}\rest_j(g^{P_m}h_{\mu(m)})\Res_j\left( dh_{\mu(1)}\wedge\cdots\wedge dh_{\mu(m-1)}\wedge\psi^{P_m}\wedge\omega_-\right).
\end{align}
Observe that the last term does not contain the functions $h_{\mu(m)+1},\dots,h_{N}$. The iterative limit $\lim_{h_{\mu(m)+1}\rightarrow{\bf 1}_{D_{\mu(m)+1}}}\circ\cdots\circ\lim_{h_N\rightarrow{\bf 1}_{D_N}}$ does not change the last term.
After taking this limit, we consider the limit $\lim_{h_{\mu(m)}\rightarrow{\bf 1}_{D_{\mu(m)}}}$.
Since $\lim_{h_{\mu(m)}\rightarrow{\bf 1}_{D_{\mu(m)}}}\rest_{j}h_{\mu(m)}=0$ almost everywhere on $D_j$ if $j\neq\mu(m)$, the remaining term after this limit is precisely
\begin{align}
&2\pi\ii(-1)^{n-1}\int_{D_{\mu(m)}}\rest_{\mu(m)}(g^{P_m})\Res_{\mu(m)}\left( dh_{\mu(1)}\wedge\cdots\wedge dh_{\mu(m-1)}\wedge\psi^{P_m}\wedge\omega_-\right)\nonumber\\
=&
2\pi\ii\int_{D_{\mu(m)}}\rest_{\mu(m)}(g^{P_m}dh_{\mu(1)}\wedge\cdots\wedge dh_{\mu(m-1)}\wedge\psi^{P_m})\wedge\Res_{\mu(m)}\omega_-\label{eqn:3.21}
\end{align}
in view of Proposition \ref{prop:3.5}.
Here, we used the fact that $\psi^{P_m}$ is smooth along $D_{\mu(m)}$.
When $m=1$, $\rest_{\mu(1)}\psi^{P_1}=\Res_{\mu(1)}(\nabla)^{-1}\Res_{\mu(1)}\omega_+$ and this implies that $\rest_{\mu(1)}\psi^{P_1}\wedge\Res_{\mu(1)}\omega_-=0$ unless $n=1$ and $\int_{D_{\mu(1)}}\rest_{\mu(1)}\psi^{P_1}\wedge\Res_{\mu(1)}\omega_-=\langle \Res_{\mu(1)}\omega_+|\Res_{\mu(1)}(\nabla)^{-1}|\Res_{\mu(1)}\omega_-\rangle$ if $n=1$.
Let us assume that $m\geq 2$.
We set $\varphi:=\rest_{\mu(m)}(g^{P_m\setminus\{ \mu(m)\}}\psi^{P_m})\wedge\Res_{\mu(m)}\omega_-$.
Note that $\varphi$ is an $(n-1,n-m)$-form on $D_{\mu(m)}$.
Then, the iterative limit $\lim_{h_{\mu(m-1)+1}\rightarrow{\bf 1}_{D_{\mu(m-1)+1}}}\circ\cdots\circ\lim_{h_{\mu(m)-1}\rightarrow{\bf 1}_{D_{\mu(m)-1}}}$ does not change (\ref{eqn:3.21}).
Let us observe that \mychange{the identity}
\begin{align}
&2\pi\ii\int_{D_{\mu(m)}}dh_{\mu(1)}\wedge\cdots\wedge dh_{\mu(m-1)}\wedge\varphi\nonumber\\
=&2\pi\ii(-1)^{m-2}\int_{D_{\mu(m)}}\left\{\dbar\left(dh_{\mu(1)}\wedge\cdots\wedge dh_{\mu(m-2)}\wedge h_{\mu(m-1)}\varphi\right)-dh_{\mu(1)}\wedge\cdots\wedge dh_{\mu(m-2)}\wedge h_{\mu(m-1)}\dbar\varphi\right\}.
\end{align}
\mychange{is true.}
Since $\varphi$ does not contain the function $h_{\mu(m-1)}$, the second term vanishes after the limit $h_{\mu(m-1)}\rightarrow{\bf 1}_{\mu(m-1)}$. On the other hand, we have
\begin{align}
&\lim_{h_{\mu(m-1)}\rightarrow{\bf 1}_{D_{\mu(m-1)}}}2\pi\ii(-1)^{m-2}\int_{D_{\mu(m)}}\dbar\left(dh_{\mu(1)}\wedge\cdots\wedge dh_{\mu(m-2)}\wedge h_{\mu(m-1)}\varphi\right)\nonumber\\
\overset{Prop. \ref{prop:ResF}}{=}&\lim_{h_{\mu(m-1)}\rightarrow{\bf 1}_{D_{\mu(m-1)}}}(2\pi\ii)^2(-1)^{n-1+m-2}\sum_{j=1,j\neq\mu(m)}^{N}\nonumber\\
&\int_{D(j,\mu(m))}\rest_{j,\mu(m)}(h_{\mu(m-1)}dh_{\mu(1)}\wedge\cdots\wedge dh_{\mu(m-2)} )\wedge\Res_j\varphi\\
=&(2\pi\ii)^2(-1)^{n+m-1}\int_{D({\mu(m-1),\mu(m)})}\rest_{\mu(m-1),\mu(m)}(dh_{\mu(1)}\wedge\cdots\wedge dh_{\mu(m-2)} )\wedge\Res_{\mu(m-1)}\varphi\\
=&(2\pi\ii)^2(-1)^{m-1}\int_{D({\mu(m-1),\mu(m)})}\nonumber\\
&\rest_{\mu(m-1),\mu(m)}(g^{P_m\setminus\{\mu(m)\}}dh_{\mu(1)}\wedge\cdots\wedge dh_{\mu(m-2)} \wedge\psi^{P_m})\wedge\Res_{\mu(m-1)}\circ\Res_{\mu(m)}\omega_-.
\end{align}
Repeating the computation above, we obtain

\begin{align}
&-\int_X\dbar(H(P_m)\wedge\psi^{P_m}\wedge\omega_-)\nonumber\\
=&(-1)^{\frac{m(m-1)}{2}}(2\pi\ii)^m\int_{D(P_m)}\rest_{P_m}\psi^{P_m}\wedge\Res_{\mu(1)}\circ\cdots\circ\Res_{\mu(m)}\omega_-\\
=&
\begin{cases}
0& (m<n)\\
(2\pi\ii)^n\langle \Res_{P_n}(\omega_+)|\Res_{P_n}(\nabla)^{-1}|\Res_{P_n}(\omega_-)\rangle& (m=n).
\end{cases}
\end{align}
Note that $\Res_{P_m}\omega_\pm$ are both holomorphic $n-m$ forms on $D(P_m)$ and their wedge product vanishes unless $m=n$.

\section{Proof of Theorem \ref{thm:2}}\label{sec:3}

We use the notation of \S\ref{sec:2.2}.
Let us consider the family version of the cohomology intersection form following the construction of \cite{LS}.
Introducing new parameters $s_1,\dots,s_m$, we write $B$ for the polynomial ring $\C[s_1,\dots,s_m]$ and write $B_{\rm loc}$ for the localization of $B$ obtained by inverting the linear polynomials $l_i(s)+j$ for any $j\in\Z$.
We set \mychange{$s=(s_1,\dots,s_m)$ and} $\nabla_{\pm}=d\pm dF_s\wedge$.
For any integer $j$, we set $\Homo^n_\pm(jD)[s]_{\rm loc}:=\mathbb{H}^n\left(X;(\Omega_{\log}^\bullet(jD)\otimes_{\C}B_{\rm loc},\nabla_{\pm})\right)$.
For any complex numbers $\alpha=(\alpha_1,\dots,\alpha_m)\in\C^m$ such that $l_i(\alpha)\notin\Z$, we have a specialization morphism ${\rm sp}(\alpha):\Homo^n_\pm(jD)[s]_{\rm loc}\rightarrow\Homo^n_\pm(jD)$ which fits into a commutative diagram
\begin{equation}
\xymatrix{
\Homo^n_{\pm}(jD)[s]_{\rm loc}\ar[r] \ar[d]^-{{\rm sp}(\alpha)} &\Homo^n_\pm((j+1)D)[s]_{\rm loc}\ar[d]^-{{\rm sp}(\alpha)}\\
\Homo^n_{\pm}(jD)\ar[r] &\Homo^n_\pm((j+1)D).
}
\end{equation}
As in the proof of \cite[2.10. Lemma]{EV}, one can show that the canonical morphism $(\Omega_{\log}^\bullet(jD)[s]_{\rm loc},\nabla_{\pm})\rightarrow (\Omega_{\log}^\bullet((j+1)D)[s]_{\rm loc},\nabla_{\pm})$ is a  quasi-isomorphism for any $j\in\Z$.
Therefore, we have an isomorphism $\Homo_\pm^n[s]_{\rm loc}\simeq\mathbb{H}^n(X;(\Omega_{\log}^\bullet(*D)\otimes_{\C}B_{\rm loc},\nabla_{\pm}))=\mathbb{H}^n(U;(\Omega_{U}^\bullet\otimes_{\C}B_{\rm loc},\nabla_{\pm}))$, from which we obtain a natural surjection $\Homo^0(U;\Omega^n_U)\otimes_{\C}B_{\rm loc}\rightarrow \Homo_\pm^n[s]_{\rm loc}$.
Obviously, the substitution $s_i=\tau \alpha_i$ induces a morphism $B_{\rm loc}\rightarrow A_{\rm loc}$ of $\C$-algebras.
The corresponding morphism of cohomology groups $\Homo_\pm^n[s]_{\rm loc}\rightarrow\Homo_\pm^n[\tau]_{\rm loc}$ is also induced.
As in \S\ref{sec:2.2}, we can define a bilinear pairing $\langle\bullet,\bullet\rangle_{ch}^s:\Homo^n_{-}[s]_{\rm loc}\otimes_{B_{\rm loc}}\Homo^{n}_{+}(-D)[s]_{\rm loc}\rightarrow B_{\rm loc}$ so thhat it fits in a commutative diagram
\begin{equation}
\xymatrix{
\Homo^n_{-}[s]_{\rm loc}\otimes_{B_{\rm loc}}\Homo^{n}_{+}(-D)[s]_{\rm loc}\ar[r]^-{\langle\bullet,\bullet\rangle_{ch}^s} \ar[d] &B_{\rm loc}\ar[d]\\
\Homo^n_{-}[\tau]_{\rm loc}\otimes_{A_{\rm loc}}\Homo^{n}_{+}(-D)[\tau]_{\rm loc}\ar[r]^-{\langle\bullet,\bullet\rangle_{ch}^\tau} \ar[d]_-{{\rm sp}(\tau_0)} &A_{\rm loc}\ar[d]^-{{\rm sp}(\tau_0)}\\
\Homo^n_{-}\otimes_{\C}\Homo^{n}_{+}(-D)\ar[r]^-{\langle\bullet,\bullet\rangle_{ch}} &\C
}.
\end{equation}
The construction above proves the following
\begin{prop}
For any $\omega_\pm\in\Homo^0(U;\Omega_U^n)$, the leading term of the Laurent expansion of $\langle\omega_-,\omega_+\rangle_{ch}^\tau$ is a rational function in $\alpha$.
\end{prop}

\noindent
Each term of the sum (\ref{eqn:2.7}) can be rewritten in terms of an integral of an $(n,n-1)$-form over a small sphere around $p$, therefore it is a local analytic function of $\alpha$ ({\it principle of continuity,} \cite[pp651--657]{GH}).
Thus, (\ref{eqn:2.7}) is an analytic function of $\alpha$ and it is enough to prove Theorem \ref{thm:2} for $\alpha\in\Q^m$ with the condition (generic).
Moreover, we are reduced to the case when the Varchenko's conjecture is true as it is a generic condition.
Re-scaling $\tau$ if necessary, we may assume that $\alpha\in\Z^m$.

Hereafter, we always equip $U$ with the analytic topology.
We simply write $U$ for $U^{an}$.
Let us follow the construction of \cite{LS}.
For the basics of the twisted homology theory, see e.g. \cite[Chapter V]{Bredon}. We only use Theorem 6.3 and Theorem 9.2 of [loc.\hspace{-.1em} cit.].
Let $\widetilde{\T}\rightarrow\T$ be the universal covering of $\T={\rm Specan}\;\C[t_1^{\pm1},\dots,t_m^{\pm1}]$ and let $\widetilde{U}$ be the cartesian product

\begin{equation}
\xymatrix{
\widetilde{U}\ar[r]^-{p} \ar[d]&U\ar[d]^-{f}\\
\widetilde{\T}\ar[r]&\T\ar@{}[lu]|{\square}.
}
\end{equation}
Using a coordinate \mychange{system} $(t_1,\dots,t_m)$ of $\T$, we identify the group ring with complex coefficients of the group of deck transformations ${\rm Deck}(\widetilde{U}/U)\simeq{\rm Deck}(\widetilde{\T}/\T)\simeq\Z^m$ with $\C[T^{\pm1}]:=\C[T_1^{\pm1},\dots,T_m^{\pm1}]$.
\mychange{Here, each $T_i$ corresponds to the counter-clockwise loop in the complex plane $\C={\rm Specan}\ \C[t_i]$.}
We set $\mathcal{L}^+:=p_!\underline{\C}_{\widetilde{U}}$.
The action of ${\rm Deck}(\widetilde{U}/U)$ to each fiber of $p$ naturally induces an action of $\C[T^{\pm1}]$ on $\mathcal{L}^+$.
The dual sheaf of $\mathcal{L}^+$ as a $\C[T^{\pm1}]$-module is denoted by $\mathcal{L}^-$.
If we symbolically write $T_i=e^{2\pi\ii s_i}$, $\mathcal{L}^+$ amounts to the local system of monodromy of a multi-valued function $f_1^{s_1}\cdots f_m^{s_m}$ with generic parameters $s_i$.
Any local monodromy of $\mathcal{L}^+$ along a component $D_i$ is given by a multiplication by a monomial $T^a$.
The condition $(*)$ guarantees that $T^a\neq 1$, i.e., $a$ is not a zero vector.
Therefore, we obtain a natural isomorphism $j_!\mathcal{L}^\pm\tilde{\rightarrow}\R j_*\mathcal{L}^\pm$.
In view of the Poincar\'e duality (\cite[Chapter V, Theorem 9.2]{Bredon}), this induces a purity of homology groups $\Homo_p(U;\mathcal{L}^\pm)=0$ ($p\neq n$) and an isomorphism $\Homo_n(U;\mathcal{L}^\pm)\tilde{\rightarrow}\Homo_n^{\rm lf}(U;\mathcal{L}^\pm)$.
We write ${\rm reg}:\Homo_n^{\rm lf}(U;\mathcal{L}^\pm)\tilde{\rightarrow}\Homo_n(U;\mathcal{L}^\pm)$ for the inverse of this natural isomorphism.

For any element $\alpha=(\alpha_1,\dots,\alpha_m)\in\C^m$, we set $\C(\alpha):=\C[T^{\pm1}]/\langle T_1-e^{2\pi\ii\alpha_1},\dots,T_m-e^{2\pi\ii\alpha_m}\rangle$.
We set $\mathcal{L}^\pm(\alpha):=\C(\alpha)\otimes_{\C[T^{\pm1}]}\mathcal{L}^\pm$.
Note that $\mathcal{L}^\pm(\alpha)$ is nothing but the sheaf of flat section of the integrable connection $\nabla_\mp$.
Using the global section $1\in\C(\alpha)$, one can define a morphism of sheaves ${\rm sp}(\alpha):\mathcal{L}^\pm\rightarrow\mathcal{L}^\pm(\alpha)$ which fits into the following commutative diagram:
\begin{equation}
\xymatrix{
\Homo_n(U;\mathcal{L}^\pm)\ar[r] \ar[d]^-{{\rm sp}(\alpha)} &\Homo_n^{\rm lf}(U;\mathcal{L}^\pm)\ar[d]^-{{\rm sp}(\alpha)}\\
\Homo_n(U;\mathcal{L}^\pm(\alpha))\ar[r] &\Homo_n^{\rm lf}(U;\mathcal{L}^\pm(\alpha)).
}
\end{equation}
When we assume that the natural morphism $\Homo_n(U;\mathcal{L}^\pm(\alpha))\rightarrow\Homo_n^{\rm lf}(U;\mathcal{L}^\pm(\alpha))$ is an isomorphism, we write ${\rm reg}(\alpha)$ for the inverse of this morphism.
Then, we have a commutativity \mychange{relation} ${\rm sp}(\alpha)\circ{\rm reg}={\rm reg}(\alpha)\circ{\rm sp}(\alpha)$.
For any element $[\Gamma]$ of $\Homo_n(U;\mathcal{L}^\pm)$ or $\Homo_n^{\rm lf}(U;\mathcal{L}^\pm)$, we write $[\Gamma(\alpha)]$ for the element ${\rm sp}(\alpha)[\Gamma]$.

Let us start the proof of Theorem \ref{thm:2}.
\mychange{In the following, we fix $\alpha\in\Z^m$ generically so that Varchenko's conjecture is true.}
Replacing $X$ by a composition of \mychange{a sequence of blowing-ups} of $X$ if necessary, we may assume that the simple normal crossing divisor $D=X\setminus U$ is decomposed as $D=D_0\cup D^\prime\cup D_\infty$ and it satisfies
\begin{enumerate}
\item The function $e^{F_\alpha}:U\rightarrow\C$ naturally extends to a morphism $X\rightarrow\mathbb{P}^1$,
\item The support of the divisors of zeros (resp. poles) of $e^{F_\alpha}$ is $D_0$ (resp. $D_\infty$)  
\end{enumerate}
(\cite[p245]{Sil}, \cite[Chapter 4,\S2]{GH}).
As we retake $X$, the condition $(*)$ is in general violated.
However, for any index $i$ such that $D_i\subset D_0\cup D_\infty$, there exist an index $j$ so that ${\rm ord}_{D_i}f_j\neq 0$.
In particular, the local monodromy of $\mathcal{L}$ along a simple loop around $D_i\subset D_0\cup D_\infty$ is given by a multiplication by a monomial $T^a$ with $a\neq 0$.
Note that $D_0=\{\Re F_\alpha=-\infty\}$ , $D_\infty=\{\Re F_\alpha=+\infty\}$ and $D_0\cap D_\infty=\varnothing$.
In view of \cite[Lemma 5.2]{Sil}, the number of critical points of $|e^{F_\alpha}|^2$ (hence that of $\Re F_\alpha$) in $U=X\setminus D$ is equal to $|{\rm Crit}(F_\alpha)|$. Moreover, \mychange{[loc. cit.] also shows that }$\Re F_\alpha$ does not have a critical point on $D^\prime\setminus (D_0\cup D_\infty)$.
Let us take a small open neighborhood $T$ of $D$ in $X$ so that both $X\setminus T$ and the closure $\bar{T}$ are manifolds with boundary, the set ${\rm Crit}(F_\alpha)$ is contained in $X\setminus T$, and $d\Re F_\alpha\neq 0$ on $\partial T$.
Now we equip $U$ with a complete Riemannian metric $g$. By \cite[Theorem A]{Smale} (see also \cite[Theorem 2.27]{Nic}), we can take a $C^1$ vector field $V$ on $X\setminus T$ close to the gradient vector field ${\rm grad}_g\Re F_\alpha$ so that $V\cdot\Re F_\alpha>0$ outside the set ${\rm Crit}(F_\alpha)$ and any pair of stable and unstable manifolds of $V$ intersects transversally.
We extend $V$ to a $C^1$ vector field on $U$ so that $V\cdot\Re F_\alpha>0$ outside the set ${\rm Crit}(F_\alpha)$.
We write $\Phi_t$ ($-\infty<t<\infty$) for the 1-parameter subgroup generated by $V$. Let us take a point $p\in{\rm Crit}(F_\alpha)$ and fix a value of $\Im F_\alpha(p)$, namely we fix a branch of $F_\alpha$ near $p$.
We set $\Gamma^\pm_p:=\{ x\in U\mid \lim_{t\rightarrow\pm\infty}\Phi_t(x)=p\}$. Note that $\Gamma_p^+\cap\Gamma_q^-=\varnothing$ if $p\neq q$ by the transversality.
We take a complex Morse coordinate $(\zeta_j)=(\xi_j+\ii\eta_j)$ so that $F_\alpha(\zeta)=F_\alpha(p)+\sum_j\zeta_j^2$ and set $c_\ve^+:=\{ \sum_j\xi^2=\ve,\eta_1=\cdots=\eta_n=0\}$ and $c_\ve^-:=\{ \sum_j\eta^2=\ve,\xi_1=\cdots=\xi_n=0\}$ for small positive numbers $\ve>0$.
Since $V$ can be taken to be gradient-like in the sense of \cite[p45]{Nic}, we may assume that $\Gamma_p^\mp=\cup_{\pm t\geq 0}\Phi_t(c_\ve^\pm)\cup\cup_{0\leq\ve^\prime\leq\ve}c_\ve^\pm$.
Since $\Re F_\alpha$ is monotonically increasing along any trajectory of $V$, $\Gamma_p^\pm$ define elements of the locally finite homology group $[\Gamma^\pm_p]\in\Homo_n^{\rm lf}(U;\mathcal{L}^\pm)$, where we identify $\Gamma_p^\mp$ with a locally finite chain $\cup_{\pm t\geq 0}\Phi_t(c_\ve^\pm)+\cup_{0\leq\ve^\prime\leq\ve}c_\ve^\pm$.

Let us construct the regularizations of $[\Gamma^\pm_p]$. For this purpose, we need a 

\begin{lem}\label{lem:tubular}
There exists a small compact neighborhood $W_\infty$ (resp. $W_0$) of $D_\infty$ (resp. $D_0$) such that $\Homo_p^{\rm lf}(W_\infty\setminus D;\mathcal{L}^-)=0$ (resp. $\Homo_p^{\rm lf}(W_0\setminus D;\mathcal{L}^+)=0$) for any $p$.
\end{lem}

\noindent
({\it Proof of Lemma \ref{lem:tubular}}) We only construct $W_\infty$ as the construction of $W_0$ is completely analogous.
Let $\widetilde{X}$ be the real oriented blowing-up of $X$ along $D$ and $\varpi:\widetilde{X}\rightarrow X$ be the associated projection (\cite[\S 8.2]{Sabbah}).
Note that $\widetilde{X}$ is naturally equipped with the structure of a manifold with corners in the sense of \cite{Joyce}.
For each component $D_i$ of $D$, \mychange{we set $\widetilde{D}_i:=\varpi^{-1}(D_i)\subset\widetilde{X}$.}
For any subdivisor $D_I=\cup_{i\in I}D_i$, we set $\widetilde{D}_I:=\cup_{i\in I}\widetilde{D}_i$.
Now, it is enough to construct a small compact neighborhood $\widetilde{W}_\infty$ of $\widetilde{D}_\infty$ so that $\Homo^p(\widetilde{W}_\infty\setminus \widetilde{D};\mathcal{L}^-)=0$.
Let $\tilde{\iota}:U\hookrightarrow\widetilde{X}$ be the natural inclusion. In view of the fact that $\widetilde{X}$ is a topological manifold with boundary such that $\partial \widetilde{X}=\widetilde{D}$, $\tilde{\iota}_*\mathcal{L}^-$ is a local system on $\widetilde{X}$.
Let $z=(z_1,\dots,z_n)$ be a coordinate neighborhood of a point of $D_\infty$ and let $((r_i,e^{\ii\theta_i})_{i=1}^b,z_{b+1},\dots,z_n)$ be the associated coordinate on $\widetilde{X}$ so that the morphism $\varpi$ is locally given by $((r_i,e^{\ii\theta_i})_{i=1}^b,z_{b+1},\dots,z_n)\mapsto ((r_ie^{\ii\theta_i})_{i=1}^b,z_{b+1},\dots,z_n)$. We may assume $\widetilde{D}_\infty=\{ r_1\cdots r_a=0\}$ and $\widetilde{D^\prime}=\{ r_{a+1}\cdots r_b=0\}$ in this neighborhood.
We can consider a local vector field $\sum_{i=1}^a\frac{\partial}{\partial r_i}$ whose stalk at each point $q$ belongs to the inward sector $IS(T_q\widetilde{X})$ in the sense of \cite[Definition 2.2]{Joyce}.
By means of a partition of unity, we can construct a vector field $\Theta$ on $\widetilde{X}$ which belongs to the inward sector $IS(T_q\widetilde{X})$ and is tangent to each component $D_i\subset D^\prime$.
Let $\Psi_t$ be the 1-parameter subgroup generated by $\Theta$.
For any $T>0$, $\bigcup_{0\leq t\leq T}\Psi_t(\widetilde{D}_\infty)$ (resp. $\bigcup_{0\leq t\leq T}\Psi_t(\widetilde{D}_\infty\setminus\widetilde{D}^\prime)$) is a deformation retract neighborhood of $\widetilde{D}_\infty$ (resp. $\widetilde{D}_\infty\setminus\widetilde{D}^\prime$).
We fix $T_0>0$ and set $\widetilde{W}_\infty:=\bigcup_{0\leq t\leq T_0}\Psi_t(\widetilde{D}_\infty)$.
Note that $\widetilde{W}_\infty\setminus \widetilde{D}$ is homotopic to $\widetilde{D}_\infty\setminus \widetilde{D}^\prime$.
Since the morphism $\varpi:\widetilde{D}_\infty\setminus \widetilde{D}^\prime\rightarrow D_\infty\setminus D^\prime$ is proper, we have an isomorphism $R^p\varpi_*(\tilde{\iota}_*\mathcal{L}^-)_x=\Homo^p(\varpi^{-1}(x);\tilde{\iota}_*\mathcal{L}^-)$ for any $x\in D_\infty\setminus D^\prime$.
In view of the fact that $\varpi^{-1}(x)$ is a product of circles and $\tilde{\iota}_*\mathcal{L}^-$ induces a non-trivial local system on it, we have the vanishing $R^p\varpi_*(\tilde{\iota}_*\mathcal{L}^-)_x$ for any $x\in D_\infty\setminus D^\prime$, hence $R^p\varpi_*(\tilde{\iota}_*\mathcal{L}^-)=0$ on $D_\infty\setminus D^\prime$ for any $p\in\Z$.
Leray's spectral sequence shows the vanishing $\Homo^p(\widetilde{W}_\infty\setminus \widetilde{D};\mathcal{L}^-)=\Homo^p(\widetilde{W}_\infty\setminus \widetilde{D};\tilde{\iota}_*\mathcal{L}^-)=\Homo^p(\widetilde{D}_\infty\setminus \widetilde{D}^\prime;\tilde{\iota}_*\mathcal{L}^-)=0$, hence the lemma is a consequence of Poincar\'e duality (\cite[Chapter V, theorem 9.2]{Bredon}).
\qed

\vspace{1em}

Since $e^{F_\alpha}$ is a proper map, $\{ \Re F_\alpha>M\}$ ($M>0$) forms a fundamental system of neighborhoods of $D_\infty$.
We take $M>\ve$ so that $\{ \Re F_\alpha>M\}\subset W_\infty:=\varpi(\widetilde{W}_\infty)$.
Now let us take $T>0$ so that $\Phi_{T}(c_\ve^+)\subset \{ \Re F_\alpha>M\}$.\footnote{The existence of such a number can easily be justified by a standard argument of a flow on a Riemannian manifold (e.g. \cite[Lemma 6.4.5]{Jost}).}
Since $\Homo_{n-1}(W_\infty\setminus D;\mathcal{L}^-)=0$, we can take a singular $n$-chain $C\in C_{n}(W_\infty\setminus D;\mathcal{L}^-)$ so that $\partial C=-\Phi_{T}(c_\ve^+)$.
We set $\widetilde{\Gamma}_p^-:=\cup_{0\leq\ve^\prime\leq\ve}c_{\ve^\prime}^++\cup_{0\leq t\leq T}\Phi_t(c_\ve^+)+C$.
By construction, we have $\partial\widetilde{\Gamma}_p^-=0$.
Moreover, $[\Gamma_p^--\widetilde{\Gamma}_p^-]=[\cup_{T\leq t}\Phi_t(c_\ve^+)-C]\in\Homo^{\rm lf}_n(W_\infty\setminus D;\mathcal{L}^+)=0$ implies that there is a locally finite chain $C^\prime$ in $W_\infty\setminus D$ so that $\Gamma_p^--\widetilde{\Gamma}_p^-=\partial C^\prime$.
Since $W_\infty\setminus D\hookrightarrow U$ is a proper map, $C^\prime$ can be regarded as a locally finite chain in $U$.
Therefore, we have ${\rm reg}[\Gamma_p^-]=[\widetilde{\Gamma}_p^-]$.
In the same manner, using a neighborhood $W_0$ of $D_0$ instead of $W_\infty$, we can construct $\widetilde{\Gamma}_p^+$ so that ${\rm reg}[\Gamma_p^+]=[\widetilde{\Gamma}_p^+]$.
Let us take $\tau>0$ to be generic so that $\Homo_p(U;\mathcal{L}^\pm(\tau\alpha))\overset{\sim}{\rightarrow}\Homo_p^{\rm lf}(U;\mathcal{L}^\pm(\tau\alpha))$ for any $p\in\Z$.
By the transversality of stable and unstable manifolds and the fact that $D_0\cap D_\infty=\varnothing$, we have $\langle[\widetilde{\Gamma}_p^-(\tau\alpha)],[\widetilde{\Gamma}_p^+(\tau\alpha)]\rangle_h=1$ and $\langle[\widetilde{\Gamma}_p^-(\tau\alpha)],[\widetilde{\Gamma}_q^+(\tau\alpha)]\rangle_h=0$ for $p\neq q$.
Therefore, $\{[\widetilde{\Gamma}_p^\pm(\tau\alpha)]\}_{p\in{\rm Crit}(F_\alpha)}$ is a set of independent elements of $\Homo_n(U;\mathcal{L}(\tau\alpha))$.
The cardinality of ${\rm Crit}(F_\alpha)$ is equal to the signed Euler characteristic $(-1)^n\chi(U)$, which is also equal to the dimension $\dim_{\C}\Homo_n(U;\mathcal{L}^\pm(\tau\alpha))$ in view of the pure codimensionality $\Homo_p(U;\mathcal{L}^\pm(\tau\alpha))=0$ ($p\neq n$).
We can conclude that $\{[\widetilde{\Gamma}_p^\pm(\tau\alpha)]\}_{p\in{\rm Crit}(F_\alpha)}$ is a basis of $\Homo_n(U;\mathcal{L}^\pm(\tau\alpha))$.

On $\cup_{0\leq t\leq T}\Phi_t(c_\ve^+)+C$, we have $\Re F_\alpha-\Re F_\alpha(p)\geq \ve$.
Moreover, we have an estimate
\begin{equation}
\left|\int_{\cup_{0\leq t\leq T}\Phi_t(c_\ve^+)+C}e^{-\tau F_\alpha}\omega_-\right|\leq A e^{-\frac{\ve}{2}\tau }
\end{equation}
for some constant $A>0$.
A standard argument of stationary phase method (e.g. \cite[Lemma 4.11]{AK}) gives rise to the asymptotic formulas 
\begin{equation}\label{eqn:asy1}
\int_{\tilde{\Gamma}_p^-(\tau\alpha)}e^{-\tau F_\alpha}\omega_-\sim e^{-\tau F_\alpha(p)} \frac{(2\pi)^{\frac{n}{2}}}{\sqrt{\det H_{F_\alpha,x}(p)}}\frac{\omega_-}{dx}\mychange{(p)}\tau^{-\frac{n}{2}}(1+o(\tau^{-1}))
\end{equation}
and 
\begin{equation}\label{eqn:asy2}
\int_{\tilde{\Gamma}_p^+(\tau\alpha)}e^{\tau F_\alpha}\omega_+\sim (\ii)^ne^{\tau F_\alpha(p)} \frac{(2\pi)^{\frac{n}{2}}}{\sqrt{\det H_{F_\alpha,x}(p)}}\frac{\omega_+}{dx}\mychange{(p)}\tau^{-\frac{n}{2}}(1+o(\tau^{-1}))
\end{equation}
as $\tau\rightarrow+\infty$.
Since the homology intersection matrix $(\langle [\tilde{\Gamma}^-_p(\tau\alpha)],[\tilde{\Gamma}^+_q(\tau\alpha)]\rangle_h)_{p,q\in{\rm Crit}(F_\alpha)}$ is an identity matrix, the twisted period relation (\ref{eqn:TPR1}) reads
\begin{equation}
\langle\omega_-,\omega_+\rangle_{ch}=\sum_{p\in{\rm Crit}(F_\alpha)}\left( \int_{\tilde{\Gamma}^-_p(\tau\alpha)}e^{-\tau F_\alpha}\omega_-\right)\left( \int_{\tilde{\Gamma}^+_p(\tau\alpha)}e^{\tau F_\alpha}\omega_+\right)
\end{equation}
for any generic $\tau$, hence we obtain
\begin{equation}\label{eqn:TPR}
\langle\omega_-,\omega_+\rangle_{ch}^\tau=\sum_{p\in{\rm Crit}(F_\alpha)}\left( \int_{\tilde{\Gamma}^-_p(\tau\alpha)}e^{-\tau F_\alpha}\omega_-\right)\left( \int_{\tilde{\Gamma}^+_p(\tau\alpha)}e^{\tau F_\alpha}\omega_+\right).
\end{equation}
Substituting (\ref{eqn:asy1}) and (\ref{eqn:asy2}) to (\ref{eqn:TPR}), we conclude that the leading term of (\ref{eqn:exp}) is given by (\ref{eqn:leading}).
\qed

\section*{Appendix: Cayley trick for dual volume}
In this appendix, we provide a combinatorial proof of the identity (\ref{eqn:CT}).  In order to fix the scale of integration variable, we fix the ambient lattice.
Let $M\simeq \Z^n$ be a free abelian group of rank $n$ and let $N=\Hom_{sg}(M,\Z)$ be its dual.
Here, the subscript $sg$ stands for semi-group homomorphisms.
Viewing the positive part of the real line $\R_{>0}$ as a (semi-)group by usual product of real numbers, we set $N_{>0}:=\Hom_{sg}(M,\R_{>0})$.
The isomorphism $\R\rightarrow\R_{>0}$ of semi-groups given by exponential naturally induces a diffeomorhism $\Exp:N_{\R}\rightarrow N_{>0}$. 
We set $M_{\R}:=M\otimes_\Z\R$ and $N_{\R}:=N\otimes_\Z\R$.
Let $p\in\R[M]$ be a polynomial with positive coefficients.
The symbol $x^\omega$ denotes the monomial in $\R[M]$ corresponding to an element $\omega\in M$.
Writing $p(x)=\sum_{\omega\in A}c_{\omega}x^{\omega}$ ($c_{\omega}>0$), we set $\Cone(A):=\sum_{\omega\in A}\R_{\geq 0}\omega\subset M_\R$, $\Delta^\prime_A:={\rm convex\ hull}\{ \omega\}_{\omega\in A}$ and $\Delta_A:={\rm convex\ hull}\{ \omega\}_{\omega\in A}\cup \{ O\}$.
We assume the following condition:
\begin{equation}
\text{$\Cone(A)$ is $n$-dimensional and pointed.}
\end{equation}
Note that we say $\Cone(A)$ is pointed if there is a dual vector $\phi\in N_\R$ such that $\langle\phi,\omega\rangle>0$ for any $\omega\in A$.
Let us recall the definition of the normalized dual volume of a cone.
Taking a set of free generators of $M$, one can define a volume form (or the canonical form) $\omega_0=\frac{dx}{x}$ of the the positive part $N_{>0}$.
In the following, we fix $\omega_0$ and equip $N_{>0}$ with an orientation $\omega_0>0$.
This choice amounts to fixing a $\Z$-linear isomorphism $\det:\bigwedge^n M\tilde{\rightarrow}\Z$ which naturally induces $\det:\bigwedge^n N\tilde{\rightarrow}\Z$.
We set $\dvol:=\Exp^*\omega_0$ and regard $\Exp$ as an orientation preserving diffeomorphism.
Consider a convex polyhedral pointed cone $C\subset M_\R$ and a vector $X\in M_\R$ which does not lie on any hyperplane spanned by a facet of $C$.
For any facet $F$ of $C$, we take a vector $\phi_F\in N_\R$ which is non-negative on $C$ and defines a facet $F$.
The dual cone $C^*_X$ is the convex polyhedral cone spanned by vectors $\{\langle\phi_F,X\rangle\phi_F\}_{F<C:facets}$ whose orientation is that of $N_\R$ multiplied by $-1$ to the power of  the number of facets of $C$ such that $\langle\phi_F,X\rangle<0$.
If $X$ is in the interior of the cone $C$, we simply write $C^*$ for $C^*_X$.
The dual normalized volume $\vol_\Z(C^*_X)$ is defined by the following equation (cf.\cite[\S 7, (7.174), (7.186)]{AHP}):
\begin{equation}
\vol(C^*_X):=\int_{C^*_X}e^{-\langle\phi,X\rangle}\dvol(\phi).
\end{equation}
Note that $\frac{\vol(C^*_X)}{n!}$ is the canonical function in the sense of \cite{AHP}.

Let us take $C=\Cone(A)$.
For any $\alpha>0$ and $X\in {\rm Int}(C)$, we set
\begin{equation}
I(\alpha):=\alpha^n\int_{N_{>0}}e^{-p(x)}x^{\alpha X}\omega_0.
\end{equation}
An argument closely related to the proof of \cite[\S 2, claim 1]{AHP} gives the following integral representation of the dual volume.
\begin{prop}\label{prop:ExponentialAmplitude}
Let ${\rm C.T.}(p)$ denote the constant term of $p$. One has a formula
\begin{equation}
\lim_{\alpha\rightarrow+0}I(\alpha)=e^{-{\rm C.T.}(p)}\vol(\Cone(A)^*_X).
\end{equation}
\end{prop}

\begin{proof}
It is enough to prove the proposition when ${\rm C.T.}(p)=0$.
Let $\Sigma$ be a complete simplicial fan in $N_{\R}$ which refines the normal fan of the polytope $\Delta_A$.
We consider a decomposition of the integration domain $N_{>0}=\bigcup_{\s\in\Sigma}\Exp(-\sigma)$.
For each cone $\s\in\Sigma$, we consider an integral
\begin{equation}
I_\s(\alpha):=\alpha^n\int_{\Exp(-\s)}e^{-p(x)}x^{\alpha X}\omega_0.
\end{equation}
We are only interested in $n$-dimensional cones $\s\in\Sigma$.
Let $\phi_1,\dots,\phi_n$ generate the cone $\s$.
We may assume that $\det(\phi_1\wedge\cdots\wedge\phi_n)>0$.
We consider a parametrization $(0,1)^n\ni\xi=(\xi_1,\dots,\xi_n)\mapsto \sum_{i=1}^n(\log \xi_i)\phi_i\in -\s$.
Setting $Z_i:=\langle\phi_i,X\rangle$, we have
\begin{equation}
I_\s(\alpha)=\det(\phi_1\wedge\cdots\wedge\phi_n)\alpha^n\int_{(0,1)^n}e^{-q(\xi)}\xi^{\alpha Z}\frac{d\xi}{\xi}
\end{equation}
for some positive Laurent polynomial $q(\xi)$.

Firstly, we consider the case when $\s\subset \Cone(A)^*$.
In this case, each $\phi_i$ belongs to the dual cone $\Cone(A)^*$ and we obtain an inequality $\min_{v\in\Delta_A^\prime}\langle\phi_i,v\rangle\geq 0$.
This implies that $q(\xi)$ is a non-constant polynomial.
Moreover, $\mychange{X}\in{\rm Int}(\Cone(A))$ shows that $Z_i>0$.
Let us decompose the integral $I_\s(\alpha)$ as
\begin{equation}
I_\s(\alpha)=\det(\phi_1\wedge\cdots\wedge\phi_n)\alpha^n\left\{\int_{(0,1)^n}\xi^{\alpha Z}\frac{d\xi}{\xi}-\int_{(0,1)^n}(1-e^{-q(\xi)})\xi^{\alpha Z}\frac{d\xi}{\xi}\right\}.
\end{equation}
Taking the limit $\alpha\rightarrow+0$, the first term clearly converges to $\frac{\det(\phi_1\wedge\cdots\wedge\phi_n)}{Z_1\cdots Z_n}=\vol_\Z(\s)$.
In view of the fact that $q(\xi)$ is a positive polynomial and that there is a function $f:(0,1)^n\rightarrow\R$ bounded and positive such that $1-e^{-q(\xi)}=q(\xi)f(\xi)$, we obtain an inequality
\begin{equation}\label{eqn:417}
(0\leq) \alpha^n\int_{(0,1)^n}(1-e^{-q(\xi)})\xi^{\alpha Z}\frac{d\xi}{\xi}
\leq C\alpha^n\int_{(0,1)^n}q(\xi)\xi^{\alpha Z}\frac{d\xi}{\xi}
\end{equation}
for some positive real number $C$.
Since $q(\xi)$ is a non-constant polynomial, the last term of (\ref{eqn:417}) is a linear combination of terms of the form $\alpha^n\int_{(0,1)^n}\xi^{a+\alpha Z}\frac{d\xi}{\xi}$ with non-negative exponents $a\neq 0$.
Clearly, each of these converges to $0$.
The computation above shows the equality $\lim_{\alpha\rightarrow+0}I_\s(\alpha)=\vol_\Z(\s)$.

It remains to prove
\begin{equation}\label{eqn:ZL}
\lim_{\alpha\rightarrow+0}I_\s(\alpha)=0
\end{equation}
when $\s\not\subset \Cone(A)^*$.
Let $1\leq r\leq n$ be a number so that $\phi_1,\dots,\phi_r\notin\Cone(A)^*$ and $\phi_{r+1},\dots,\phi_n\in\Cone(A)^*$.
We have relations $\min_{v\in\Delta_A^\prime}\langle\phi_i,v\rangle<0$ for $i=1,\dots,r$ and $\min_{v\in\Delta_A^\prime}\langle\phi_i,v\rangle=0$ and $Z_i>0$ for $i=r+1,\dots,n$.
Therefore, the Laurent polynomial $q$ can be written in a form $q(\xi)=c\xi^a(1+r(\xi))$ where $c$ is a positive real number, $a=(a_1,\dots,a_r,0,\dots,0)$ is an exponent vector such that $a_i<0$ and $r(\xi)$ is a positive polynomial.
In view of the inequality $e^{-q(\xi)}\leq e^{-c\xi^a}$ on $(0,1)^n$, (\ref{eqn:ZL}) follows from the next lemma.
\end{proof}

\begin{lem}
For any real vector $Z\in \R^n$, one has
\begin{equation}\label{eqn:ZL2}
\lim_{\alpha\rightarrow+0}\alpha^n\int_{(0,1)^n}e^{-\xi_1^{-1}\cdots\xi_n^{-1}}\xi^{\alpha Z}\frac{d\xi}{\xi}=0.
\end{equation}
\end{lem}

\begin{proof}
We set $g(\zeta):=\zeta^{-1}e^{-\zeta^{-1}}$ ($\zeta\in (0,1)$).
The integral in the limit (\ref{eqn:ZL2}) takes the form
\begin{equation}
\alpha^n\int_{(0,1)^n}g(\xi_1\cdots\xi_n)\xi^{\alpha Z}d\xi_1\cdots d\xi_n.
\end{equation}
When $\alpha$ is small enough, $\xi^{\alpha Z}$ is integrable.
Since $g$ is a bounded function, the lemma follows from Lebesgue's dominance convergence theorem. 
\end{proof}

Let us argue that Proposition \ref{prop:ExponentialAmplitude} can be seen as an extension of \cite[\S2, claim 1]{AH}.
For this purpose, let us assume that there is an element $\phi\in N$ such that $\langle\phi,\omega\rangle=1$ for any $\omega\in A$.
We can regard $\Delta^\prime_A$ as a full-dimensional convex polytope in the affine space $\phi^{-1}(1)\subset M_\R$.
Let us take a point $X$ from the interior of $\Delta^\prime_A$.
We also regard $X\in\Hom_{sg}(N,\Z)$.
The subspace $X^{-1}(1)\subset N_\R$ is naturally equipped with a positive volume element $\eta$ such that $\dvol=d\phi\wedge\eta$.
By abuse of notation, we also write $\vol_{\Z}$ for the volume measured by $(n-1)!\eta$.
We set $P_A^*:=C_X^*\cap X^{-1}(1)$.
This is compatible with the usual definition of the dual polytope: Identification $X^{-1}(1)\ni\psi\mapsto \psi-\phi\in\Ker X$.
$P_A^*=\{\psi\in\Ker X\mid \langle\psi,p-X\rangle\geq -1,\forall p\in P_A\}$.
By construction, we have an equality $\vol_\Z(C^*_X)=\vol_{\Z}(P^*_X)$.

More explicitly, we can take a split $M=\Z\oplus\Ker \phi$ to write each exponent $\omega\in A$ and $X$ as $\omega=(1,\omega^\prime)$ and $X=(1,u)$ ($\omega^\prime,u\in\Ker\phi$).
When $X=(v,u)$ does not lie on the hyperplane $\phi^{-1}(1)$ but still belongs to ${\rm Int}\Cone(A)$, one puts $P:=vP_A$ and one has the equality $v\vol_\Z(\Cone(A)^*_X)=\vol_\Z((P-u)^\circ)$.

One can also confirm the equality $v\vol_\Z(\Cone(A)^*_X)=\vol_\Z((P-u)^\circ)$ through integration.
The following equation is easily obtained from the integral representation of gamma function $\Gamma(v)=\int_{\R_{>0}}e^{-y}y^v\frac{dy}{y}$ and the change of a variable $y\rightarrow yp(x)$:
\begin{equation}\label{eqn:CTI}
\int_{N_{>0}}p(x)^{-v}x^u\frac{dx}{x}=\frac{1}{\Gamma(v)}\int_{\R_{>0}\times N_{>0}}e^{-yp(x)}y^{v}x^u\frac{dy}{y}\frac{dx}{x}.
\end{equation}
Thus, we recover \cite[\S2, claim 1]{AH}:
\begin{equation}
\int_{N_{>0}}p(x)^{-v}x^u\frac{dx}{x}=\vol_\Z((P-u)^\circ).
\end{equation}

One can also treat a weighted Minkowski sum of Newton polytopes in the same manner.
Let $q_i=\sum_{\omega^{(i)}\in A_i}c_{\omega^{(i)}}x^{\omega^{(i)}}\in\R[M]$ ($i=1,\dots,e$) be positive Laurent polynomials.
If $\{ {\bf e}_i\}_{i=1}^e$ denotes the standard basis of $\Z^e\subset\R^e$, we set $A:=\cup_{i=1}^e\{({\bf e}_i,\omega^{(i)})\}_{\omega^{(i)}\in A_i}\subset\Z^e\oplus M=:\hat{M}$.
For simplicity, we assume that $A$ generates the lattice $\hat{M}$.
Repeated applications of formulas similar to (\ref{eqn:CTI}) yield the following theorem which proves the identity (\ref{eqn:CT}).
\begin{thm}\label{prop:A2}
Let $v_1,\dots,v_e$ be positive real numbers.
We set $P:=\sum_{i=1}^ev_i{\rm New}(q_i)\subset M_{\R}$.
For any $u\in M_\R$, $u$ lies in the interior of $P$ if and only if $X:=(v,u)\in\hat{M}_{\R}$ lies in the interior of the cone $\Cone(A)$.
In this case, one has the following equalities: 
\begin{align}
v_1\cdots v_e\vol_\Z(\Cone(A)_X^*)&=\lim_{\alpha\rightarrow+0}\alpha^{n+e}\int_{\R_{>0}^e\times N_{>0}}e^{-\sum_{i=1}^ey_iq_i(x)}y^{\alpha v}x^{\alpha u}\frac{dy}{y}\wedge\frac{dx}{x}\\
&=\lim_{\alpha\rightarrow+0}\alpha^{n}\int_{N_{>0}}\prod_{i=1}^eq_i(x)^{-\alpha v_i}x^{\alpha u}\frac{dx}{x}\\
&=\vol_\Z((P-u)^\circ)
\end{align}
Moreover, if $T$ is a triangulation of $\Cone(A)$ such that $X$ does not lie on any facet of any simplex $\s\in T$, one has
\begin{equation}\label{eqn:DE}
\vol_\Z((P-u)^\circ)=v_1\cdots v_e\sum_{\s\in T}\frac{1}{\prod_{i\in\s}p_{\s i}(X)}\frac{1}{|\det A_\s|}.
\end{equation}
\end{thm}

\begin{proof}
We only need to prove the identity
\begin{equation}\label{eqn:DE2}
\vol_\Z(\Cone(A)_X^*)=\sum_{\s\in T}\frac{1}{\prod_{i\in\s}p_{\s i}(X)}\frac{1}{|\det A_\s|}.
\end{equation}
Regarding $p_{\s i}$ as an element of $\Hom_{\R}(\hat{M},\R)$, it is straightforward to see that the identity
\begin{equation}
\vol_\Z(\s^*_X)=\frac{1}{\prod_{\in\s}p_{\s i}(X)}\frac{1}{|\det A_\s|}
\end{equation}
is true.
Note that $\det (\bigwedge_{i\in\s} p_{\s i})=\pm \frac{1}{\det A_\s}$.
Thus, (\ref{eqn:DE2}) is a consequence of Filliman duality (\cite{F}, \cite{K})
\begin{equation}
\vol_\Z(\Cone(A)^*_X)=\sum_{\s\in T}\vol_\Z(\s^*_X).
\end{equation}
\end{proof}

\end{document}